\documentclass[11pt,twoside, a4paper]{amsart}

\usepackage[english]{babel} 
\usepackage[T1]{fontenc} 
\usepackage{amsmath}
\usepackage{amssymb} 
\usepackage{float}
\usepackage{amsfonts}
\usepackage{mathtools}
\usepackage[utf8]{inputenc}
\usepackage[mathcal]{eucal} 
\usepackage{enumerate}
\usepackage{amsthm} 
\usepackage{hyperref}
\usepackage[totalwidth=14cm,totalheight=20cm, hmarginratio=1:1]{geometry}

\newtheorem{defi}{Definition}[section] 

\newtheorem{cor}[defi]{Corollary} 
\newtheorem{lemma}[defi]{Lemma}
\newtheorem{prop}[defi]{Proposition}
\newtheorem{oss}[defi]{Remark}

\newtheorem{bigthm}{Theorem}

\newcommand{\Pp}{\mathbb{P}}
\newcommand{\R}{\mathbb{R}} 

\newcommand{\C}{\mathbb{C}} 
\newcommand{\Z}{\mathbb{Z}}

\newcommand{\h}{\mathbb{H}}

\newcommand{\diag}{\mathrm{diag}}

\newcommand{\SL}{\mathrm{SL}}

\newcommand{\dPSL}{\mathbb{P}SL(2,\mathbb{R})\times \mathbb{P}SL(2,\mathbb{R})}

\newcommand{\PSL}{\mathbb{P}SL}

\newcommand{\Isom}{\mathrm{Isom}}

\newcommand{\Id}{\mathrm{Id}}
\newcommand{\trace}{\mathrm{trace}}

\newcommand{\SO}{\mathrm{SO}}
\newcommand{\Aut}{\mathrm{Aut}}
\newcommand{\graph}{\mathrm{graph}}
\newcommand{\cro}{\mathrm{cr}}
\newcommand{\Sp}{\mathrm{Sp}}
\newcommand{\dive}{\mathrm{div}}
\newcommand{\grad}{\mathrm{grad}}

\DeclareMathAlphabet{\mathpzc}{OT1}{pzc}{m}{it}

\title[Quadratic differentials and light-like polygons]{Polynomial quadratic differentials on the complex plane and light-like polygons in the Einstein Universe}
\author{Andrea Tamburelli}
\date{\today}
\thanks{}

\begin{document}

\begin{abstract} 
We construct geometrically a homeomorphism between the moduli space of polynomial quadratic differentials on the complex plane and light-like polygons in the 2-dimensional Einstein Universe. As an application, we find a class of minimal Lagrangian maps between ideal polygons in the hyperbolic plane.
\end{abstract}

\maketitle
\setcounter{tocdepth}{1}
\tableofcontents

\section*{Introduction} 
\indent A general problem in Teichm\"uller theory consists in finding canonical maps between hyperbolic surfaces. Different possibilities are known when the surfaces are closed: in the isotopic class of the identity, we can find, for instance, the Teichm\"uller map that minimises the quasi-conformal dilatation (\cite{Teichmap}), the harmonic map that minimises the $L^{2}$-energy (\cite{eells}, \cite{Wolf_harmonic}, \cite{Tromba_book}), and the minimal Lagrangian map that realises the minimum of the holomorphic $1$-energy (\cite{oneharmonic}).\\
\indent One aim of this paper is to study minimal Lagrangian maps between ideal polygons in the hyperbolic plane. Recall that a diffeomorphism $m:U \rightarrow V$ between domains in the hyperbolic plane is minimal Lagrangian if it preserves the volume and its graph inside $\h^{2}\times \h^{2}$ is a minimal surface. These maps can equivalently be characterised by the fact that they can be decomposed as $m=f'\circ f^{-1}$, where $f:X \rightarrow U$ and $f':X \rightarrow V$ are harmonic maps from a Riemann surface $X$ with opposite Hopf differentials (\cite{Schoenharmonic}). We will also require that the  metrics $\|\partial f\|^{2}|dz|^{2}$ and $\|\partial f'\|^{2}|dz|^{2}$ induced on $X$ by the harmonic maps are complete. We obtain the following:

\begin{bigthm}\label{thm:A} Given two ideal polygons in the hyperbolic plane with $k \geq 3$ vertices, there exist at most $k$ minimal Lagrangian maps that factor through the complex plane, sending one polygon to the other.
\end{bigthm}

\indent Notice that the condition on the number of vertices is necessary for the existence of a map that preserves the volume. The assumption on the completeness of the metrics on $\C$ is technical and might be removed provided every harmonic diffeomorphism from the complex plane to an ideal polygon has polynomial Hopf differential. Here we consider two minimal Lagrangian maps to be different if one cannot be obtained from the other by pre- or post-composition by a global isometry of the hyperbolic plane. The different minimal Lagrangian maps correspond to different couplings between vertices of the polygon in the domain and edges of the polygon in the target: in fact, the minimal Lagrangian maps that we find cannot be extended to the boundary of the polygon, as they send a neighbourhood of each ideal vertex in the domain to a neighbourhood of an edge in the target, but one can prescribe to which edge a vertex is associated. This produces at most $k$ minimal Lagrangian maps, because if any of the two polygons has some symmetries, these maps can be obtained by pre- or post-composing by a global isometry. Therefore, the minimal Lagrangian maps that we describe behave very differently compared to the ones that Brendle found between smooth domains with strictly convex boundary, as in that case it is possible to choose the image of an arbitrary point of the boundary. (\cite{brendle2008}). \\
\\
\indent Theorem \ref{thm:A} will be proved using tools coming from anti-de Sitter geometry. In fact, minimal Lagrangian maps between domains in the hyperbolic plane are intimately related to maximal space-like surfaces in the three-dimensional anti-de Sitter space with given boundary at infinity. See for instance \cite{bon_schl} and \cite{seppimaximal}. (See also \cite{volumeAdS}, \cite{entropy}, \cite{regularAdS}, \cite{wildAdS}, \cite{degeneration} for applications.) It turns out that minimal Lagrangian maps between ideal polygons in the hyperbolic plane factoring through the complex plane correspond to maximal surfaces in anti-de Sitter space bounding a light-like polygon at infinity, i.e. a topological circle consisting of a finite number of light-like segments. Recall, namely, that the boundary at infinity of anti-de Sitter space is naturally endowed with a conformally flat Lorentzian structure and is a model for the $2$-dimensional Einstein Universe $Ein^{1,1}$. We prove the following:

\begin{bigthm}\label{thm:B} Given a light-like polygon $\Delta \subset Ein^{1,1}$, there exists a unique maximal surface with boundary at infinity $\Delta$.
\end{bigthm}

\indent This extends previous results about existence and uniqueness of maximal surfaces with given boundary at infinity (see for instance \cite{bon_schl}, \cite{bbsads}, and \cite{TambuCMC} for a generalisation to constant mean curvature surfaces), and we believe it may have an independent interest. These surfaces have a special feature: they are conformally equivalent to the complex plane. To the extend of our knowledge, these are the first such examples, if we exclude the trivial case of the horospherical surface described in \cite{bon_schl} and \cite{seppimaximal}. By associating to every such surface, the holomorphic quadratic differential that determines its second fundamental form, we obtain the following:

\begin{bigthm}\label{thm:C} There is a homeomorphism between the moduli space of polynomial quadratic differentials on the complex plane and the moduli space of light-like polygons in the Einstein Universe.
\end{bigthm}

\indent This can be thought of as analogous to the homeomorphism between the moduli space of polynomial cubic differentials on the complex plane and convex polygons in the real projective plane, found in \cite{DW}. As in the aforementioned paper, this result has an interpretation in terms of Higgs bundles with wild ramifications: compactifying $\C$ with $\C\Pp^{1}$, we can see a holomorphic quadratic differential $q$ as a meromorphic quadratic differential on $\C\Pp^{1}$ with a pole of order at least $2$ at infinity. From these data, we can construct a parabolic $\dPSL$-Higgs bundle on $\C\Pp^{1}$ with Higgs field (determined by $q$) carrying an irregular singularity at infinity. The solution of Hitchin's equation in this context (\cite{BB_wild}) produces a minimal surface in $\h^{2} \times \h^{2}$, which is the image of the maximal surface with holomorphic quadratic differential $q$ provided by Theorem \ref{thm:B} via the Gauss map. Our theorem thus describes the geometry of these surfaces, as being asymptotic to a finite number of flats in $\h^{2}\times \h^{2}$, the number being determined by the degree of the polynomial. In an upcoming joint work with Mike Wolf (\cite{TW}), we extend this picture to $\Sp(4,\R)$-Higgs bundles with wild ramifications, in which case the role of the quadratic differential is played by a quartic differential.

\subsection*{Outline of the paper} In Section \ref{sec:polyno} we recall some basic facts about polynomial quadratic differentials on the complex plane. In Section \ref{sec:polygo} we describe the geometry of the moduli space of light-like polygons in the Einstein Universe. We then relate these two moduli spaces via maximal surfaces in anti-de Sitter space in Section \ref{sec:from_polyno}. Section \ref{sec:mainthm} is devoted to the proofs of Theorem \ref{thm:B} and \ref{thm:C}. Theorem \ref{thm:A} is proved in Section \ref{sec:application}.

\subsection*{Acknowledgement}
This project was inspired by a collaboration with Mike Wolf on a similar subject. I would like to thank him for the many interesting conversations and his hospitality. The author acknowledges support from U.S. National Science Foundation grants DMS 1107452, 1107263, 1107367 "RNMS: GEometric structures And Representation varieties"(the GEAR Network).

\section{Polynomial quadratic differentials}\label{sec:polyno}
A polynomial quadratic differential is a holomorphic differential on the complex plane of the form $p(z)dz^{2}$, where $p(z)$ is a polynomial function. 

\subsection{The moduli space} We denote with $\mathcal{Q}_{d}$ the space of polynomial quadratic differentials of degree $d$. The group $\Aut(\C)$ of biholomorphisms of $\C$ acts on this space by push-forward. Let $\mathcal{MQ}_{d}$ be the quotient of $\mathcal{Q}_{d}$ by this action. The geometry of the resulting moduli space is analogous to that described for polynomial cubic differentials in \cite{DW}.

\begin{prop}\label{prop:moduli_poly}The moduli space $\mathcal{MQ}_{d}$ is a complex orbifold of real dimension $2(d-1)$ if $d\geq 1$.
\end{prop}
\begin{proof}Every polynomial quadratic differential can be written as
\[
	q=(a_{d}z^{d}+a_{d-1}z^{d-1}+\cdots +a_{0})dz^{2} \ ,
\]
for some $a_{i} \in \C$ and $a_{d} \in \C^{*}$. An element $T(z)=bz+c \in \Aut(\C)$ acts on $q$ via
\[
	T_{*}q=(a_{d}b^{d+2}(z+c/b)^{d}+a_{d-1}b^{d+1}(z+c/b)^{d-1}+\cdots +b^{2}a_{0})dz^{2} \ ,
\]
hence by choosing $b=a_{d}^{-1/(d+2)}$ we can make it monic (i.e. with leading coefficient equal to $1$) and a suitable choice of the translation component $c$ can make it centered (i.e. with $a_{d-1}=0$). Notice that these choices are unique up to multiplying $b$ by a $(d+2)$-root of unity. Thus we can describe the moduli space as the quotient
\[
	\mathcal{MQ}_{d}=\mathcal{TQ}_{d}/\Z_{d+2}
\]
where $\mathcal{TQ}_{d}$ is the space of monic polynomial of degree $d$ whose roots sum to $0$ and $\Z_{d+2}$ denotes the cyclic group of order $d+2$ generated by $T(z)=\zeta_{d+2}z$, for a primitive $(d+2)$-root of unity $\zeta_{d+2}$. Since $\mathcal{TQ}_{d}$ is naturally identified with $\C^{d-1}$ by
\begin{align*}
	&\mathcal{TQ}_{d} \rightarrow \C^{d-1}\\
	(z^{d}+a_{d-2}&z^{d-2}+\cdots+a_{0}) \mapsto (a_{d-2}, \cdots, a_{0}) \ ,
\end{align*}
it follows that $\mathcal{MQ}_{d}$ is a complex orbifold of real dimension $2(d-1)$.
\end{proof}
    
\begin{oss}If $d=0$, the space $\mathcal{MQ}_{0}$ consists of only one point, represented by the quadratic differential $q=dz^{2}$.
\end{oss}

\noindent We put on $\mathcal{MQ}_{d}$ the topology induced by the identification 
\[
    \mathcal{MQ}_{d}\cong \R^{2(d-1)}/\Z_{d+2}
\]
found in Proposition \ref{prop:moduli_poly}. The following remark will be useful in the rest of the paper:
\begin{prop}\label{prop:poly_top} Let $[q_{n}] \in \mathcal{MQ}_{d}$ be a sequence of polynomial quadratic differentials. The following facts are equivalent:
\begin{enumerate}
    \item [i)] $[q_{n}]$ converges to $[q]$ in $\mathcal{MQ}_{d}$;
    \item [ii)] there exists a sequence $A_{n}$ of biholomorphisms of $\C$ such that $(A_{n})_{*}q_{n}$ converges uniformly on compact sets to $q$.
\end{enumerate}
\end{prop}
\begin{proof}For the first part of the proof we suppose that $q_{n}$ are monic and centered representatives of $[q_{n}]$. If $[q_{n}]$ converges to $[q]$ in $\mathcal{MQ}_{d}$, then, denoting with $T(z)=\zeta_{d+2}z$ the generator of the $\Z_{d+2}$-action, the coefficients of $T_{*}^{j_{n}}q_{n}$ converge to the coefficients of $q$ for some $j_{n} \in \{1, \dots, d+2\}$. This clearly implies that $T_{*}^{j_{n}}q_{n}$ converges to $q$ uniformly on compact sets.\\
\indent Viceversa, suppose that $(A_{n})_{*}q_{n}$ converges to $q$ uniformly on compact sets. Then, since $(A_{n})_{*}q_{n}$ is a sequence of holomorphic polynomials, it actually converges analytically. In particular, evaluating the sequence and its derivatives at zero, we deduce that the coefficients of $(A_{n})_{*}q_{n}$ must converge to the coefficients of $q$. Let $B_{n}, B \in \Aut(\C)$ be biholomorphisms of $\C$ such that $(B_{n}A_{n})_{*}q_{n}$ and $B_{*}q$ are monic and centered. In the proof of Proposition \ref{prop:moduli_poly}, it is shown that the linear and the translation parts of $B_{n}$ and $B$ can be written explicitly in terms of the coefficients of the polynomials, with the only ambiguity given by the choice of a $(d+2)$-root of unity. Hence, after fixing such a root of unity, we can conclude that if we write $B_{n}(z)=b_{n}z+c_{n}$ and $B=bz+c$, then $b_{n} \to b$ and $c_{n}\to c$. Therefore, since the action of biholomorphisms of $\C$ on the coefficients of polynomial quadratic differentials is continuous, the coefficients of the monic and centered representatives $(A_{n}B_{n})_{*}q_{n}$ and $B_{*}q$ converge.
\end{proof}

\subsection{Half-planes and rays} A natural coordinate for a quadratic differential $q$ is a local coordinate $w$ on a open subset of $\C$ in which $q=dw^{2}$. Such a coordinate always exists locally away from the zeros of $q$, because near such a point we can choose a holomorphic root of $q$ and define
\[
	w(z)=\int_{z_{0}}^{z} \sqrt{q} \ .
\]
Any two natural coordinates for $q$ differ by a multiplication by $-1$ and an additive constant. The metric $|q|$ defines a flat structure on $\C$ with singularities at the zeros of $q$: a zero of order $k$ corresponds to a cone point of angle $\pi(k+2)$. \\
\\
\indent We can see $q$ as a meromorphic quadratic differential on the Riemann sphere $\C\Pp^{1}=\C\cup\{\infty\}$: at the point at infinity $q$ carries a pole singularity of order at least $2$. A natural set of coordinates for $q$ in a neighbourhood of infinity is described in \cite[Section 10.4]{Strebel}. We recall it briefly here.\\
\indent A $q$-half-plane is a pair $(U,w)$ where $U\subset \C$ is open and $w$ is a natural coordinate for $q$ that maps $U$ diffeomorphically to the right half-plane $\{\Re(w)>0\}$. Given any monic polynomial quadratic differential $q$ of degree $d\geq 1$, it is possible to find a system of $(d+2)$ coordinate charts $\{(U_{i}, w_{i})\}_{i=1, \dots, d+2}$ with the following properties:
\begin{itemize}
	\item [i)] the complement of $\bigcup \overline{U_{i}}$ is pre-compact;
	\item [ii)] each $(U_{i},w_{i})$ is a $q$-half-plane;
	\item [iii)] $w_{i}(\overline{U_{i}} \cap \overline{U_{i+1}})$ and $w_{i+1}(\overline{U_{i}} \cap \overline{U_{i+1}})$ are half-lines contained in $i\R$;
	\item [iv)] $\overline{U_{i}}\cap \overline{U_{j}}=\emptyset$ if $i\ne j\pm 1$ \ .
\end{itemize}

\indent We will also make use of the following terminology. A path in $\C$ whose image in a natural coordinate for $q$ is a Euclidean ray with angle $\theta$ will be called $q$-ray with angle $\theta$. This means that in a natural coordinate, a $q$-ray is of the form $\gamma(t)=b+e^{i\theta}t$. We will call $b$ the height of the ray. Similarly a $q$-quasi-ray with angle $\theta$ is a path that can be parameterised in a natural coordinate as $\gamma(t)=e^{i\theta}t+o(t)$ as $t \to +\infty$.

\section{Light-like polygons in the Einstein Universe}\label{sec:polygo}
The $2$-dimensional Einstein Universe is topologically a torus endowed with a conformally flat Lorentzian structure. We concretely realise it as the conformal boundary at infinity of the $3$-dimensional anti-de Sitter space.

\subsection{The Einstein Universe} Consider the vector space $\R^{4}$ endowed with the bilinear form of signature $(2,2)$
\[
	\langle x, y \rangle= x_{0}y_{0}+x_{1}y_{1}-x_{2}y_{2}-x_{3}y_{3} \ .
\]
We denote 
\[
	\widehat{AdS}_{3}=\{x \in \R^{4} \ | \ \langle x,x\rangle=-1\} \ .
\]
It can be easily verified that $\widehat{AdS}_{3}$ is diffeomorphic to a solid torus and the restriction of the bilinear form to the tangent space at each point endows $\widehat{AdS}_{3}$ with a Lorentzian metric of constant sectional curvature $-1$. Anti-de Sitter space is then defined as
\[
	AdS_{3}=\Pp(\{x \in \R^{4}\ | \ \langle x ,x \rangle<0\}) \subset \R\Pp^{3} \ .
\]
The natural map $\pi:\widehat{AdS}_{3} \rightarrow AdS_{3}$ is a two-sheeted covering and we endow $AdS_{3}$ with the induced Lorentzian structure. The isometry group of $AdS_{3}$ that preserves the orientation and time-orientation is $\SO_{0}(2,2)$, the connected component of the identity of the group of linear transformations that preserve the bilinear form of signature $(2,2)$.\\
\\
\indent The boundary at infinity of anti-de Sitter space is naturally identified with 
\[
	\partial_{\infty}AdS_{3}=\Pp(\{ x \in \R^{4} \ | \ \langle x,x\rangle =0\}) \ .
\]
It coincides with the image of the Segre embedding $s:\R\Pp^{1}\times \R\Pp^{1} \rightarrow \R\Pp^{3}$, and thus, it is foliated by two families of projective lines, which we distinguish by calling $s(\R\Pp^{1} \times \{*\})$ the right-foliation and $s(\{*\}\times \R\Pp^{1})$ the left-foliation. The action of an isometry extends continuously to the boundary, and preserves the two foliations. Moreover, it acts on each line by a projective transformation, thus giving an identification between $\SO_{0}(2,2)$ and $\dPSL$.\\
\\
\indent The Lorentzian metric on $AdS_{3}$ induces on $\partial_{\infty}AdS_{3}$ a conformally flat Lorentzian structure. To see this, notice that the map
\begin{align*}
	F: D \times S^{1} &\rightarrow \widehat{AdS_{3}}\\
		(z,w) &\mapsto \left(\frac{2}{1-\|z\|^{2}}z, \frac{1+\|z\|^{2}}{1-\|z\|^{2}}w\right)
\end{align*}
is a diffeomorphism, hence $D\times S^{1}$ is a model for anti-de Sitter space if endowed with the pull-back metric
\[
	F^{*}g_{AdS_{3}}=\frac{4}{(1-\|z\|^{2})^{2}}|dz|^{2}-\left(\frac{1+\|z\|^{2}}{1-\|z\|^{2}}\right)^{2}d\theta'^{2} \ .
\]
Therefore, by composing with the projection $\pi:\widehat{AdS_{3}}\rightarrow AdS_{3}$, we deduce that $\pi\circ F$ continuously extends to a homeomorphism
\begin{align*}
	\partial_{\infty}F: S^{1} \times S^{1} &\rightarrow \partial_{\infty}AdS_{3}\\
		(z,w) &\rightarrow (z,w)
\end{align*}
and in these coordinates the conformally flat Lorentzian structure is induced by the conformal class 
\[
	c=[d\theta^{2}-d\theta'^{2}]  \ .
\]
The Einstein Universe $Ein^{1,1}$ is the boundary at infinity of anti-de Sitter space endowed with this conformal Lorentzian structure. Notice that the light-cone at each point $p \in Ein^{1,1}$ is generated by the two lines in the left- and right-foliation described above.

\subsection{The moduli space of light-like polygons in $Ein^{1,1}$}
A light-like polygon in $Ein^{1,1}$ is an embedded, non-homotopically trivial $1$-simplex $\Delta \subset Ein^{1,1}$ homeomorphic to a circle, such that every edge is a light-like segment. We will always assume that the polygon is oriented and its orientation is compatible with the orientation of the boundary at infinity of a totally geodesic space-like plane in $AdS_{3}$. We denote with $\mathcal{MLP}_{2k}$ the moduli space of light-like polygons in $Ein^{1,1}$ with $2k$-vertices, up to the conformal action of $\dPSL$.\\
\\
\indent In order to describe the geometry of this moduli space, we recall that a curve $\gamma$ in $Ein^{1,1}$ can be seen as a graph of a function $f_{\gamma}:\R\Pp^{1} \rightarrow \R\Pp^{1}$ in the following way. Fix a totally geodesic space-like plane $P_{0}$ in $AdS_{3}$. ($P_{0}$ will be fixed for the rest of the paper.) Its boundary at infinity describes a circle in $Ein^{1,1}$. Any $\xi \in Ein^{1,1}$ lies in a unique line belonging to the right foliation and a unique line belonging to the left foliation of $Ein^{1,1}$. Those two lines intersect the boundary at infinity of $P_{0}$ in exactly one point, that we denote with $\pi_{r}(\xi)$ and $\pi_{l}(\xi)$, respectively. We can thus associate to $\gamma$ a map $f_{\gamma}:\R\Pp^{1} \rightarrow \R\Pp^{1}$ defined by the property that
\[
	f_{\gamma}(\pi_{l}(\xi))=\pi_{r}(\xi) \ \ \ \ \ \  \text{for every $\xi \in \gamma$} .
\]
This procedure gives a well-defined map, as soon as $\gamma$ is an acausal curve, i.e. locally any two points of $\gamma$ cannot be connected by a non-space-like curve (\cite{bon_schl}). However, in case of light-like polygons, we can make this construction work and associate to every light-like polygon $\Delta$ a unique upper-semicontinuous, orientation-preserving, piece-wise constant map $f_{\Delta}:\R\Pp^{1} \rightarrow \R\Pp^{1}$ that determines uniquely the polygon. Namely, if $e_{r}$ is an open edge of the polygon that lies on a line of the right-foliation, then $\pi_{l}(e_{r})$ is an open interval in the boundary at infinity of $P_{0}$ and $\pi_{r}(e_{r})$ consists of only one point. Moreover, the union of $\pi_{l}(e_{r})$ over all such open edges covers the whole boundary at infinity of $P_{0}$ but a finite number of points, which correspond to the left-projections of the vertices. If we identify $\R\Pp^{1}$ with $\partial_{\infty}P_{0}$ (and we fix such an identification from now on), there exists a unique upper-semicontinuous map $f_{\Delta}:\R\Pp^{1} \rightarrow \R\Pp^{1}$ such that
\[
	f_{\Delta}(\pi_{l}(e_{r}))=\pi_{r}(e_{r})
\]
for every open edge $e_{r}$ as above. It can be easily checked that $f_{\Delta}$ preserves the orientation. Clearly, this function determines the polygon uniquely. Notice that the conformal action of $\dPSL$ on $Ein^{1,1}$ translates into a pre- and post-composition by projective transformations.\\
\\
\indent A marked ideal polygon in $\h^{2}$ is an ideal polygon with a preferred choice of a vertex. We denote with $\mathcal{TP}_{k}$ the moduli space of marked ideal polygons in the hyperbolic plane with $k$ vertices.  

\begin{prop}\label{prop:moduli_polygo} The moduli space $\mathcal{MLP}_{2k}$ is an orbifold of dimension $2k-6$, if $k\geq 3$.
\end{prop}
\begin{proof}We define a homeomorphism from $(\mathcal{TP}_{k}\times \mathcal{TP}_{k})/\Z_{k}$ to $\mathcal{MLP}_{2k}$, where $\Z_{k}$ is the cyclic group of order $k$ corresponding to the diagonal change of markings. Using the correspondence between light-like polygons up to conformalities and upper-semicontinuous, orientation-preserving, piece-wise constant maps $f:\R\Pp^{1} \rightarrow \R\Pp^{1}$ up to pre- and post-composition with projective transformations, we only need to show that such functions are uniquely determined by two marked ideal polygons up to the diagonal action of the cyclic group. Now, given two marked ideal polygons $(P,p_{0})$ and $(Q,q_{0})$, the orientation on $\R\Pp^{1}$ induces a natural labelling of the vertices $p=p_{0}, p_{1}, \dots, p_{k-1}$ of $P$ and $q=q_{0}, q_{1}, \dots q_{k-1}$ of $Q$. We can thus construct a unique upper-semicontinuous, orientation-preserving, piece-wise constant function $f:\R\Pp^{1} \rightarrow \R\Pp^{1}$ with points of discontinuity $\{p_{0}, \dots, p_{k-1}\}$ such that
$f(p_{j})=q_{j}$ for $j=0, \dots, k-1$. Since a diagonal change of marking produces the same function, the moduli space of such maps is exactly 
$(\mathcal{TP}_{k}\times \mathcal{TP}_{k})/\Z_{k}$. \\
\indent Moreover, since projective transformations act transitively on triples of points, we have that
\[
    \mathcal{TP}_{k}\cong \R^{k-3}\ .
\]
Namely, there exists a unique representative of $(P,p_{0})$ in its $\PSL(2,\R)$-orbit such that $p_{0}=[0,1]$, $p_{1}=[1,1]$ and $p_{2}=[1,0]$ and a system of coordinates is then given for instance by the cross ratios $x_{j}=\cro(p_{0}, p_{1}, p_{2}, p_{j})$.
\end{proof}

\begin{oss} If $k=2$, the moduli space $\mathcal{MLP}_{4}$ consists of only one element, represented by the boundary at infinity of the horospherical surface in $AdS_{3}$ (see Section \ref{sec:from_polyno}).
\end{oss}

\noindent We put on $\mathcal{MLP}_{2k}$ the topology induced by the identification
\[
    \mathcal{MLP}_{2k}\cong (\mathcal{TP}_{k} \times \mathcal{TP}_{k})/\Z_{k}
\]
found in Proposition \ref{prop:moduli_polygo}. In particular, a sequence of light-like polygons $\Delta_{n}$ converges to $\Delta$ if and only if, denoting with $f_{\Delta_{n}}$ and $f_{\Delta}$ the corresponding defining functions, the points of discontinuity and the images of $f_{\Delta_{n}}$ converge to the points of discontinuity and the images of $f_{\Delta}$ preserving the markings, up to the action of $\PSL(2, \R)$. This is equivalent to say that the graphs of $f_{\Delta_{n}}$ converge to the graph of $f_{\Delta}$ in the Hausdorff topology.\\

\indent Comparing with the result of Proposition \ref{prop:moduli_poly}, one easily sees that the two moduli spaces $\mathcal{MQ}_{d}$ and $\mathcal{MLP}_{2(d+2)}$ are abstractly homeomorphic. In the next section, we will construct geometrically an explicit homeomorphism between them.

\section{From polynomial quadratic differentials to light-like polygons}\label{sec:from_polyno}
In this section we see how to associate a light-like polygon in the Einstein Universe to a polynomial quadratic differential $q$ on the complex plane. The construction is based on the existence of a complete maximal surface with second fundamental form determined by $q$. Its boundary at infinity will be the desired light-like polygon.

\subsection{Complete maximal space-like surfaces in $AdS_{3}$} We first recall some basic facts about complete space-like embeddings of surfaces in anti-de Sitter space. The material covered here is classical and can be found for instance in \cite{bbsads} and \cite{bon_schl}. See also \cite{TambuCMC} for generalisations to constant mean curvature surfaces and \cite{BTT} for higher signature. \\
\\
\indent Let $U\subset \C$ be a simply-connected domain. We say that $f:U \rightarrow AdS_{3}$ is a space-like embedding if $f$ is an embedding and the induced metric $I=f^{*}g_{AdS}$ is Riemannian. The Fundamental Theorem of surfaces embedded in anti-de Sitter space ensures that such a space-like embedding is uniquely determined, up to post-composition by a global isometry of $AdS_{3}$, by its induced metric $I$ and its shape operator $B:TU \rightarrow TU$, which satisfy 
\[
	\begin{cases}
	d^{\nabla}B=0 \ \ \ \ \ \ \ \ \ \ \ \ \ \ \ \ \ \ \ \ \ \ \ \text{(Codazzi equation)}\\
	K_{I}=-1-\det(B) \ \ \ \ \ \ \ \ \ \ \text{(Gauss equation)}
	\end{cases}
\]
where $\nabla$ is the Levi-Civita connection and $K_{I}$ is the curvature of the induced metric on $S=f(U)$. We will always assume in this paper that the induced metric $I$ is complete.\\
\\
\indent We say that $S$ is maximal if $B$ is traceless. In this case, the Codazzi equation implies that the second fundamental form $II=I(B\cdot, \cdot)$ is the real part of a quadratic differential $q$, which is holomorphic for the complex structure compatible with the induced metric $I$ on $S$. \\
\\
\indent We denote with $\widehat{S}$ a lift of $S$ to $\widehat{AdS}_{3}$. We have the following:
\begin{prop}\label{prop:graph} If we identify $\widehat{AdS}_{3}$ with $D\times S^{1}$, then $\widehat{S}$ is the graph of a $2$-Lipschitz map from $D$ to $S^{1}$. 
\end{prop}
\begin{proof}Let $\pi_{1}:\widehat{S}\subset D\times S^{1} \rightarrow D$ denote the projection onto the first factor, and let $g_{\h^{2}}$ be the hyperbolic metric on the unit disc $D$. Then, $\pi_{1}^{*}g_{\h^{2}}\geq I$. Since $I$ is complete, $\pi_{1}^{*}g_{\h^{2}}$ is also a complete Riemannian metric on $\widehat{S}$. If follows that $\pi:\widehat{S}\rightarrow D$ is a proper immersion, hence a covering. Since $D$ is simply connected and $\widehat{S}$ is connected, it is a diffeomorphism.\\
\indent Since the projection onto the first factor is a diffeomorphism, $\widehat{S}$ is the graph of a map $f:D\rightarrow S^{1}$. Because $\widehat{S}=\graph(f)$ is supposed to be space-like, for every $z \in D$ and $v \in T_{z}D$ we must have
\[
	\frac{4}{(1-\|z\|^{2})^{2}}\|v\|^{2}-\left(\frac{1+\|z\|^{2}}{1-\|z\|^{2}}\right)^{2}\|df_{z}(v)\|^{2}>0 \ ,
\]
which implies
\[
	\|df_{z}(v)\|<\frac{2}{1+\|z\|^{2}}\leq 2 \ ,
\]
hence $f$ is $2$-Lipschitz.
\end{proof}

\begin{oss} This shows that $S$ has at most two lifts to $\widehat{AdS}_{3}$, each of them diffeomorphic to a disc.
\end{oss}

\begin{cor} The closure of a complete space-like surface $S$ intersects the Einstein Universe in a topological circle.
\end{cor}
\begin{proof} By Proposition \ref{prop:graph}, a lift $\widehat{S}$ of $S$ is the graph of a $2$-Lipschitz map $f:D\rightarrow S^{1}$. This extends to a unique continuous function $\partial f:S^{1}\rightarrow S^{1}$, whose graph is thus the boundary at infinity of $\widehat{S}$. Its image under the projection $\pi:\R^{4}\setminus \{0\} \rightarrow \R\Pp^{3}$ is the boundary at infinity of $S$, which is therefore a topological curve in the Einstein Universe.
\end{proof}

\indent We will see in the next section that the boundary at infinity of a complete maximal surface whose second fundamental form is the real part of a polynomial quadratic differential on the complex plane is a light-like polygon. 

\begin{defi} Given a curve $\Gamma \subset Ein^{1,1}$, the convex hull $\mathcal{C}(\Gamma)$ of $\Gamma$ is the smallest convex subset of $AdS_{3}$ with boundary at infinity $\Gamma$.
\end{defi}

\begin{prop}\label{prop:convexhull}Let $S$ be a complete maximal surface in $AdS_{3}$ with boundary at infinity $\Gamma$. Then $S$ in contained in the convex hull of $\Gamma$.
\end{prop}
\begin{proof}By definition $S$ is a saddle surface, that is a surface which has opposite principal curvatures at each point. A characterization of saddle surfaces (\cite[Section 6.5.1]{saddle}) says that for any relatively compact subset $C\subset S$, C is contained in the convex hull of $\partial C$. This property applied to an exhaustion by compact subsets of $S$ gives the desired result.
\end{proof}

\begin{defi}Given a curve $\Gamma \subset Ein^{1,1}$, the domain of dependence $\mathcal{D}(\Gamma)$ of $\Gamma$ consists of the set of points $p \in AdS_{3}$ such that the dual plane $p^{*}$ does not intersect $\Gamma$, where the dual plane is obtained by projecting to $AdS_{3}$ the orthogonal to a lift $\hat{p} \in \widehat{AdS}_{3}$.  
\end{defi}

\begin{cor} The lift of $S$ to $\widehat{AdS}_{3}$ has exactly two connected components, each one of them homeomorphic to a disc.
\end{cor}
\begin{proof} Let $\pi:\widehat{AdS}_{3} \rightarrow AdS_{3}$ be the canonical projection. The map $\pi^{-1}(S) \rightarrow S$ is a two-sheeted covering. Given a point $x \in \pi^{-1}(S)$, the function
\begin{align*}
    \partial_{\infty}\pi^{-1}(S) &\rightarrow \{-1, 1\}\\
        \xi &\mapsto \frac{\langle x, \xi \rangle}{|\langle x, \xi \rangle|}
\end{align*}
is a well-defined continuous map, because $S$ is contained in the convex hull of its boundary at infinity, which is included in the domain of dependence. Since $\langle x, -\xi \rangle=-\langle x, \xi \rangle$, the function takes both values, thus $\pi^{-1}(S)$ has two connected components. By Proposition \ref{prop:graph}, each connected component is homeomorphic to a disc.
\end{proof}

\subsection{Polynomial maximal surfaces} We say that a complete maximal surface in anti-de Sitter space is a polynomial maximal surface if its second fundamental form is the real part of a polynomial quadratic differential. This section is devoted to the proof of the following:

\begin{prop}\label{prop:existence}Let $q$ be a polynomial quadratic differential on the complex plane. Then there exists a polynomial maximal surface $S$ embedded in $AdS_{3}$ with second fundamental form $2\Re(q)$. Moreover, $S$ is unique up to global isometries of $AdS_{3}$.
\end{prop}

The proof relies on the fact that it is always possible to find a maximal embedding $f:\C \rightarrow AdS_{3}$ with induced metric $I=2e^{2u}|dz|^{2}$ and shape operator $B=I^{-1}\Re(2q)$, for a suitable smooth function $u:\C \rightarrow \R$. Let us illustrate the procedure to construct such a maximal embedding.\\
\\
\indent Since it is convenient to work in complex coordinates, we consider $\R^{4}\subset \C^{4}$ and we extend the $\R$-bilinear form of signature $(2,2)$ to the hermitian product on $\C^{4}$ given by 
\[
	\langle z, w\rangle=z_{1}\bar{w}_{1}+z_{2}\bar{w}_{2}-z_{3}\bar{w}_{3}-z_{4}\bar{w}_{4} \ .
\]
Given a maximal conformal embedding $f:\C \rightarrow AdS_{3}$, with a slightly abuse of notation, we will still denote with $f:\C \rightarrow \widehat{AdS}_{3}\subset \C^{2,2}$ one of its lifts. Let $N$ be the unit normal vector field such that $\{f_{z}, f_{\bar{z}}, N, f\}$ is an oriented frame in $\C^{2,2}$. We define
\[
	q=\langle N_{z}, f_{\bar{z}} \rangle \ , 
\]
where $f_{\bar{z}}$ denotes the derivative of $f$ with respect to the vector field $\frac{\partial}{\partial \bar{z}}$. The embedding being maximal implies that $q$ is a holomorphic quadratic differential on the complex plane. We then define the function $u:\C \rightarrow \R$ by the relation
\[
	\langle f_{z}, f_{z} \rangle=\langle f_{\bar{z}}, f_{\bar{z}} \rangle = e^{2u} \ .
\]
Notice that, in this way, $I=2e^{2u}|dz|^{2}$ is the induced metric on the surface $S=f(\C)$. The vectors
\[
	\sigma_{1}=\frac{f_{z}}{e^{u}}, \ \ \ \sigma_{2}=\frac{f_{\bar{z}}}{e^{u}}, \ \ \ N, \ \ \ \text{and} \ \ \ f
\]
thus give a unitary frame of $(\C^{4}, \langle \cdot, \cdot \rangle)$ at every point $z \in \C$. Taking the derivatives of the fundamental relations
\[
	\langle N, N \rangle=\langle f,f\rangle=-1 \ \ \ \langle \sigma_{j}, N \rangle=\langle \sigma_{j}, f \rangle=0 \ \ \ \langle N_{z}, f_{\bar{z}}\rangle =q \ \ \ \langle \sigma_{j}, \sigma_{j}\rangle=1 \ , 
\]
one deduces that
\[
	N_{\bar{z}}=e^{-u}\bar{q}\sigma_{1} \ \ \ \bar{\partial}\sigma_{1}=-u_{\bar{z}}\sigma_{1}+e^{u}f \ \ \ \text{and} \ \ \ \bar{\partial}\sigma_{2}=u_{\bar{z}}\sigma_{2}+\bar{q}e^{-u}N \ .
\]
Therefore, the pull-back of the Levi-Civita connection $\nabla$ of $(\C^{4}, \langle \cdot, \cdot \rangle)$ via $f$ can be written in the frame $\{\sigma_{1}, \sigma_{2}, N, f\}$ as
\[
	f^{*}\nabla=Vd\bar{z}+Udz=\begin{pmatrix}
			-u_{\bar{z}} & 0 & e^{-u}\bar{q} & 0 \\
				0 & u_{\bar{z}} & 0 & e^{u} \\
				0 & e^{-u}\bar{q} & 0 & 0 \\
				e^{u} & 0 & 0 & 0 
			\end{pmatrix} d\bar{z}+ \begin{pmatrix}
						u_{z} & 0 & 0 & e^{u} \\
						0 & -u_{z} & qe^{-u} & 0 \\
						qe^{-u} & 0 & 0 & 0 \\
						0 & e^{u} & 0 & 0 
						\end{pmatrix}dz \ .
\]
and the flatness of $f^{*}\nabla$ is equivalent to $u$ being a solution of the vortex equation
\begin{equation}\label{eq:PDE}
	u_{z\bar{z}}-\frac{1}{2}e^{2u}+\frac{1}{2}e^{-2u}|q|^{2}=0 \ ,
\end{equation}
where $|q|$ denotes here the modulus of the polynomial $q(z)$ so that $q=q(z)dz^{2}$.
Viceversa, given a holomorphic quadratic differential $q$ and a solution $u$ to Equation (\ref{eq:PDE}), the $\mathfrak{so}(2,2)$-valued $1$-form
\[
	\omega=Udz+Vd\bar{z}
\]
can be integrated to a map $F:\C \rightarrow \SL(4,\C)$, whose last column gives us a maximal embedding into $AdS_{3}$ with induced metric $I=2e^{2u}|dz|^{2}$ and shape operator $B=I^{-1}\Re(2q)$. \\
\\
The proof of Proposition \ref{prop:existence} follows then from the Fundamental Theorem of surfaces embedded in $AdS_{3}$ and well-known results about vortex equations:
\begin{prop}[\cite{DW}, \cite{QL_vortex}]\label{prop:vortex} Given a polynomial quadratic differential $q$ on the complex plane, there exists a unique solution $u:\C \rightarrow \R$ to Equation (\ref{eq:PDE}). Moreover, $u$ satisfies the following estimates: 
\begin{itemize} 
    \item [i)] there exist constants $C,M>0$ depending continuously on the coefficients of $q$ such that $    \max\{-M,\frac{1}{2}\log(|q|)\} \leq u \leq \frac{1}{2}\log(|q|+C)$;
    \item [ii)] for $|z|\to +\infty$ we have $u-\frac{1}{2}\log(|q|)=O(e^{-2\sqrt{2}r}/\sqrt{r})$, where $r$ is the $|q|$-distance from the zeros of $q$. 
\end{itemize}
\end{prop}
\noindent It follows immediately that the induced metric on $S$ is complete, since
\[
	I=2e^{2u}|dz|^{2} \geq 2|q|
\]
and the complex plane endowed with the flat metric with cone singularities $|q|$ is complete.

\subsection{The horospherical surface}\label{subsec:horo} The solution to Equation (\ref{eq:PDE}) can be written explicitly in the special case when $q$ is a constant polynomial quadratic differential, and the associated maximal surface in $AdS_{3}$ appears in the literature as the horospherical surface (\cite{bon_schl}, \cite{seppimaximal}, \cite{TambuCMC}). Let us describe in detail the related frame field $F_{0}:\C \rightarrow \SL(4,\C)$ in this case.\\ 
\\
\indent We choose a global coordinate $z$ so that $q=dz^{2}$. The corresponding solution to Equation (\ref{eq:PDE}) is then clearly $u=0$. The $\mathfrak{so}(2,2)$-valued $1$-form becomes 
\[
	\omega_{0}=V_{0}d\bar{z}+U_{0}dz=\begin{pmatrix}
			0 & 0 & 1 & 0 \\
				0 & 0 & 0 & 1 \\
				0 & 1 & 0 & 0 \\
				1 & 0 & 0 & 0 
			\end{pmatrix} d\bar{z}+ \begin{pmatrix}
						0 & 0 & 0 & 1 \\
						0 & 0 & 1 & 0 \\
						1 & 0 & 0 & 0 \\
						0 & 1 & 0 & 0 
						\end{pmatrix}dz \ .
\]
The frame field of the horospherical surface is thus 
\[
	F_{0}(z)=A_{0}\exp(U_{0}z+V_{0}\bar{z}) \ ,
\]
for some constant matrix $A_{0} \in \SL(4, \C)$. For our convenience, we choose
\[
    A_{0}=\frac{1}{\sqrt{2}}\begin{pmatrix}
            1 & 1 & 0 & 0 \\
            -i & i & 0 & 0 \\
            0 & 0 & 1 & 1 \\
            0 & 0 & -1 & 1
            \end{pmatrix}
\]
A simple computation shows that the matrix $U_{0}z+V_{0}\bar{z}$ is diagonalisable by a constant unitary matrix $R$ so that
\[
	R^{-1}(U_{0}z+V_{0}\bar{z})R=\diag(2\Re(z), 2\Im(z), -2\Re(z), -2\Im(z)) \ .
\]
Therefore, we can write 
\[
	F_{0}(z)=A_{0}R\diag(e^{2\Re(z)}, e^{2\Im(z)}, e^{-2\Re(z)}, e^{-2\Im(z)})R^{-1} \ . 
\]
The resulting maximal embedding is given by the last column of $F_{0}(z)$, that is
\[
    f_{0}(z)=\frac{1}{\sqrt{2}}(\sinh(2\Re(z)), \sinh(2\Im(z)), \cosh(2\Re(z)), \cosh(2\Im(z))) \ .
\]
In particular we can compute explicitly the boundary at infinity of $f_{0}$: it a light-like polygon with $4$-vertices, as Table (\ref{table:1}) shows.
Moreover, notice that the frame field $F_{0}(z)$ being diagonalizable by a constant matrix $R$ (independent of $z$) implies that the horospherical surface $f_{0}(\C)$ is the orbit in $AdS_{3}$ of a Cartan subgroup of $\SO_{0}(2,2)$.

\medskip
\begin{table}[!htb]
\begin{center}
\begin{tabular}{l l l}
\hline
\textbf{Type of path $\gamma$} & \textbf{Direction $\theta$} & \textbf{Projective limit $v_\gamma$ of $f_{0}(\gamma)$}\\
\hline
Quasi-ray& $\theta \in (-\tfrac{\pi}{4}, \tfrac{\pi}{4})$ & $v_\gamma = [1,0,1,0]$\\
Ray (of height $iy$) & $\theta = \tfrac{\pi}{4}$ & $v_\gamma=[1,s,1,s] $ for some $s(y) \in \R$\\
& & $(v_\gamma \to [0,1,0,1]$ as $y \to \infty)$\\
Quasi-ray& $ \theta \in (\tfrac{\pi}{4}, \tfrac{3\pi}{4})$ & $v_\gamma = [0,1,0,1]$\\
Ray (of height $iy$) & $\theta = \tfrac{3\pi}{4}$ & $v_\gamma=[-s,1,s,1]$ for some $s(y) \in \R$\\
& & $(v_\gamma \to [-1,0,1,0]$ as $y \to \infty)$\\
Quasi-ray& $\theta\in (\tfrac{3\pi}{4}, \tfrac{5\pi}{4})$  & $v_\gamma = [-1,0,1,0]$\\
Ray (of height $iy$) & $\theta = \tfrac{5\pi}{4}$ & $v_\gamma=[-1,-s,1,s]$ for some $s(y) \in \R$\\
& & $(v_\gamma \to [0,-1,0,1]$ as $y \to \infty)$\\
Quasi-ray& $\theta\in (\tfrac{5\pi}{4}, \tfrac{7\pi}{4})$  & $v_\gamma = [0,-1,0,1]$\\
Ray (of height $iy$) & $\theta = \tfrac{7\pi}{4}$ & $v_\gamma=[s,-1,s,1]$ for some $s(y) \in \R$\\
& & $(v_\gamma \to [1,0,1,0]$ as $y \to \infty)$\\
\hline\\
\end{tabular}
\end{center}
\caption{Limits of the standard horospherical surface along rays}\label{table:1}
\end{table}

\subsection{The boundary at infinity of a polynomial maximal surface}\label{subsec:rays}
Let $(U,w)$ be a standard half-plane for a polynomial quadratic differential $q$ of degree $k$. Let $F:U \rightarrow \SL(4,\C)$ be the frame field for the corresponding polynomial maximal surface $f:\C \rightarrow AdS_{3}$. We define the osculating map $G:U \rightarrow \SO_{0}(2,2)$ by
\[
    G(w)=F(w)F_{0}^{-1}(w)
\]
where $F_{0}:U \rightarrow \SL(4,\C)$ denotes the frame field of the horospherical surface. Notice that the map actually takes value in $\SO_{0}(2,2)$ because both frames $F(w)$ and $F_{0}(w)$ lie in the same right coset of $\SO_{0}(2,2)$ within $\SL(4,\C)$.\\
\indent Evidently $G$ is constant if and only if $f$ is itself a horospherical surface, and more generally, left multiplication by $G(w_{0})$ transforms the standard horospherical surface described in Section \ref{subsec:horo} to one which has the same tangent plane and unit normal at the point $f(w_{0})$. Therefore, in some sense $G(w)$ represents the osculating horospherical surface of $f$ at $w_{0}$. A computation using the structure equations for a maximal surface shows that
\[
    G^{-1}dG=F_{0}\Theta F_{0}^{-1} \ ,
\]
where 
\[
	\Theta(w)=\begin{pmatrix} 
				-u_{\bar{w}} & 0 & e^{-u}-1 & 0 \\
				0 & u_{\bar{w}} & 0 & e^{u}-1 \\
				0 & e^{-u}-1 & 0 & 0 \\
				e^{u}-1 & 0 & 0 & 0 
				\end{pmatrix}d\bar{w} \ + 
\]
\[
				 \ \ \ \ \ \ \ \begin{pmatrix}
				u_{w} & 0 & 0 & e^{u}-1 \\
				0 & -u_{w} & e^{-u}-1 & 0 \\
				e^{-u}-1 & 0 & 0 & 0 \\
				0 & e^{u}-1 & 0 & 0 
				\end{pmatrix}dw
\]
and $2e^{2u}|dw|^{2}$ is the induced metric on $S=f(\C)$. \\
\indent Notice that the estimates in Proposition \ref{prop:vortex} show that $\Theta(w)$ is rapidly decaying to $0$ as the distance from $w$ to the zeros of $q$ increases. Ignoring the conjugation by the matrix $F_{0}(w)$, this suggests that $G(w)$ should approach a constant as $w$ goes to infinity, which would mean that the maximal surface $S$ is asymptotic to a horospherical surface. However, the frame field $F_{0}(w)$ is itself exponentially growing as $w$ goes to infinity, with a precise rate depending on the direction. Thus the actual asymptotic behaviour of $G$ depends on the comparison between the growth of the error $\Theta(w)$ and the frame field $F_{0}(w)$. In most directions, the exponential decay of $\Theta(w)$ is faster than the growth of $F_{0}(w)$, giving a well-defined limiting horospherical surface. In exactly $2(k+2)$ unstable directions there is an exact balance, which allow the horospherical surface to shift. We will thus prove the following:

\begin{prop}\label{prop:general}Let $q$ be a monic polynomial quadratic differential on the complex plane of degree $k \geq 1$. The polynomial maximal surface $S$ in anti-de Sitter space with second fundamental form $2\Re(q)$ is asymptotic to $2(k+2)$ horospherical surfaces. Moreover, its boundary at infinity is a light-like polygon in $Ein^{1,1}$ with $2(k+2)$ vertices.
\end{prop}

The first step of the proof consists in finding the stable directions:
\begin{defi} We say that a ray $\gamma(t)=e^{i\theta}t+y$ is stable if the direction $\theta \notin \{-\pi/4, \pi/4\}$. A $q$-quasi-ray is stable if the associated ray is stable. 
\end{defi}
Notice that the possible directions of stable rays in a standard half-plane form three open intervals
\[
J_{-}=(-\pi/2, -\pi/4) \ \ \ \ J_{0}=(-\pi/4, \pi/4) \ \ \ \ \text{and} \ \ \ \ J_{+}=(\pi/4, \pi/2) .
\]
The stability of (quasi-)rays in these directions refers to the convergence of the osculating map:
\begin{lemma}If $\gamma$ is a stable ray or quasi-ray, then $\lim_{t \to +\infty}G(\gamma(t))$ exists. Furthermore, among all such rays only three limits are achieved: there exist $L_{0}, L_{\pm} \in \SO_{0}(2,2)$ such that
\[
    \lim_{t \to +\infty}G(\gamma(t))=\begin{cases}
                    L_{+} \ \ \ \ \ \ \text{if} \ \ \theta \in J_{+} \\
                    L_{0} \ \ \ \ \ \ \text{if} \ \ \theta \in J_{0} \\
                    L_{-} \ \ \ \ \ \ \text{if} \ \ \theta \in J_{-}
                    \end{cases}
\]
\end{lemma}
\begin{proof}First we consider rays, and at the end we show that quasi-rays have the same behaviour. Let $\gamma$ be a ray. For brevity we denote $G(t)=G(\gamma(t))$. We know that
\[
    G(t)^{-1}G'(t)=F_{0}(\gamma(t))\Theta_{\gamma(t)}(\dot{\gamma}(t))F_{0}^{-1}(\gamma(t)) \ .
\]
Since $F_{0}(w)=A_{0}R\diag(e^{2\Re(w)}, e^{2\Im(w)}, e^{-2\Re(w)}, e^{-2\Im(w)})R^{-1}$, for constant matrices $R$ and $A_{0}$, the asymptotic behaviour of $F_{0}(\gamma(t))\Theta_{\gamma(t)}(\dot{\gamma}(t))F_{0}^{-1}(\gamma(t))$ depends only on the action by conjugation by the diagonal matrix
\[
    D(t)=\diag(e^{2\Re(\gamma(t))}, e^{2\Im(\gamma(t))}, e^{-2\Re(\gamma(t))}, e^{-2\Im(\gamma(t))}) \ .
\]
A direct computation shows that $R^{-1}\Theta R$ is equal to 
\[ 
    \frac{e^{-i\theta}}{4}\scalebox{0.75}{$\begin{pmatrix}
    2(e^{u}+e^{-u}-2) & -2u_{\bar{w}}+(1-i)(e^u-e^{-u}) & 0 & -2u_{\bar{w}}+(1+i)(e^u-e^{-u}) \\
    -2u_{\bar{w}}-(1-i)(e^u-e^{-u}) & 2i(e^u+e^{-u}-2) & -2u_{\bar{w}}+(1+i)(e^u-e^{-u}) & 0 \\
    0 & -2u_{\bar{w}}-(1+i)(e^u-e^{-u}) & -2(e^{u}+e^{-u}-2) & -2u_{\bar{w}}-(1-i)(e^u-e^{-u})\\
    -2u_{\bar{w}}-(1+i)(e^u-e^{-u}) & 0 & -2u_{\bar{w}}+(1-i)(e^u-e^{-u}) & -2i(e^{u}+e^{-u}-2)
    \end{pmatrix}dt$} + 
\]
\[
     \frac{e^{i\theta}}{4}\scalebox{0.75}{$\begin{pmatrix}
    2(e^{u}+e^{-u}-2) & 2u_{w}-(1+i)(e^u-e^{-u}) & 0 & 2u_{w}-(1-i)(e^u-e^{-u}) \\
    2u_{w}+(1+i)(e^u-e^{-u}) & 2i(e^u+e^{-u}-2) & 2u_{w}-(1-i)(e^u-e^{-u}) & 0 \\
    0 & 2u_{w}+(1-i)(e^u-e^{-u}) & -2(e^{u}+e^{-u}-2) & 2u_{w}+(1+i)(e^u-e^{-u})\\
    2u_{w}+(1-i)(e^u-e^{-u}) & 0 & 2u_{w}-(1+i)(e^u-e^{-u}) & -2i(e^{u}+e^{-u}-2)
    \end{pmatrix}dt$}
\]
and conjugating by $D(t)$ multiplies the $(i,j)$-entry by
\[
    \lambda_{ij}=\exp\left(2t \left(\cos\left(\theta+\frac{(i-1)\pi}{2}\right)+\cos\left(\theta+\frac{(j-1)\pi}{2}\right)\right)\right)=O\left(e^{c(\theta)t}\right) \ , 
\]
where $c(\theta)$ achieves its maximum $2\sqrt{2}$ at $\theta=\pm\pi/4$ (here we have considered only the pairs $(i,j)$ so that $\lambda_{ij}$ multiplies a non-zero entry). Combining the bounds for $R^{-1}\Theta R$ and $\lambda_{ij}$, we find that for every stable ray, 
\[
    G(t)^{-1}G'(t)=O\left(\frac{e^{-\beta t}}{\sqrt{t}}\right)
\]
where $\beta=2\sqrt{2}-c(\theta)>0$. It is then standard to show that the limit $\lim_{t \to +\infty}G(t)$ exists (\cite[Lemma B.1]{DW}). \\
\indent Now suppose that $\gamma_{1}$ and $\gamma_{2}$ are stable rays with respective angles $\theta_{1}$ and $\theta_{2}$ that belong to the same interval. For any $t\geq 0$, let $\eta_{t}(s)=(1-s)\gamma_{1}(t)+s\gamma_{2}(t)$ be the constant-speed parameterisation of the segment from $\gamma_{1}(t)$ to $\gamma_{2}(t)$. Let $g_{t}(s)=G(\eta_{t}(0))^{-1}G(\eta_{t}(s))$, which satisfies
\[
    \begin{cases} g_{t}^{-1}(s)g'_{t}(s)=F_{0}(\eta_{t}(s))\Theta_{\eta_{t}(s)}(\dot{\eta}_{t}(s))F_{0}^{-1}(\eta_{t}(s)) \\
                    g_{t}(0)=Id \\
                    g_{t}(1)=G_{1}(t)^{-1}G_{2}(t)
    \end{cases} \ ,
\]
where $G_{i}(t)=G(\gamma_{i}(t))$ for $i=1,2$. Since $|\dot{\eta}_{t}(s)|=O(t)$, the analysis above shows that 
\[
    g_{t}^{-1}(s)g_{t}'(s)=O\left(\sqrt{t}e^{-\beta t}\right) \ ,
\] 
where $\beta=2\sqrt{2}-\sup\{ c(\theta) \ | \ \theta_{1} \leq \theta \leq \theta_{2} \}$. In particular, by making $t$ large we can arrange $g_{t}^{-1}(s)g_{t}'(s)$ to be uniformly small for all $s \in [0,1]$. ODE methods (\cite[Lemma B.1]{DW}) ensure that
\[
    g_{t}(1)=G_{1}^{-1}(t)G_{2}(t) \to Id \ \ \ \text{as} \ \ t \to +\infty
\]
This shows that $G$ has the same limit along $\gamma_{1}$ and $\gamma_{2}$. \\
\indent Finally, suppose that $\gamma_{1}$ is a stable quasi-ray, and $\gamma_{2}$ is the ray that it approximates. If we consider as above $\eta_{t}(s)=(1-s)\gamma_{1}(t)+s\gamma_{2}(t)$, we have an even stronger bound on the derivative $|\dot{\eta}_{t}(s)|=o(\sqrt{t})$, hence we can conclude as before that $G(\gamma_{1}(t))$ has the same limit as $G(\gamma_{2}(t))$.
\end{proof}

Next we analyse the behaviour near unstable rays in order to understand the relationship between $L_{\pm}$ and $L_{0}$.

\begin{lemma}\label{lm:unipotent} Let $L_{\pm}$ and $L_{0}$ be as in the previous lemma. Then there exist unipotent matrices $U_{\pm}$ such that
\[
    L_{-}^{-1}L_{0}=A_{0}RU_{-}R^{-1}A_{0}^{-1} \ \ \ \ \ \text{and} \ \ \ \ \ L_{0}^{-1}L_{+}=A_{0}RU_{+}R^{-1}A_{0}^{-1}
\]
\end{lemma}
\begin{proof} We give the detailed proof for $L_{-}^{-1}L_{0}$, for the other case we only underline the differences at the end. Consider the rays $\gamma_{-}(t)=e^{-i\pi/3}t$ and $\gamma_{0}(t)=t$. By the previous lemma $G_{-}(t)=G(\gamma_{-}(t))$ and $G_{0}(t)=G(\gamma_{0}(t))$ have limit $L_{-}$ and $L_{0}$, respectively. For any $t>0$, we join $\gamma_{-}(t)$ and $\gamma_{0}(t)$ by an arc
\[
    \eta_{t}(s)=e^{is}t \ , \ \ \ \text{where} \ s \in [-\pi/3, 0] \ .
\]
Let $g_{t}(s)=G(\eta_{t}(-\pi/3))^{-1}G(\eta_{t}(s))$. Then $g_{t}:[-\pi/3, 0]\rightarrow \SO_{0}(2,2)$ satisfies the differential equation
\begin{equation}\label{eq:ODE}
    \begin{cases} g_{t}^{-1}(s)g'_{t}(s)=F_{0}(\eta_{t}(s))\Theta_{\eta_{t}(s)}(\dot{\eta}_{t}(s))F_{0}^{-1}(\eta_{t}(s)) \\
                    g_{t}(-\pi/3)=Id \\
                    g_{t}(0)=G_{-}(t)^{-1}G_{0}(t)
    \end{cases} \ .
\end{equation}  
Unlike the previous case, the coefficient  
\[
    M_{t}(s)=D(\eta_{t}(s))R^{-1}\Theta_{\eta_{t}(s)}(\dot{\eta}_{t}(s))RD(\eta_{t}(s))^{-1} 
\]
is not exponentially small in $t$ throughout the interval. At $s=-\pi/4$, conjugation by $D(\eta_{t}(-\pi/4))$ multiplies the $(1,2)$-entry and the $(4,3)$-entry by a factor $\exp(2\sqrt{2}t)$, exactly matching the decay rate of $R^{-1}\Theta R$ and giving
\[
    M_{t}(-\pi/4)=O\left(\frac{|\dot{\eta}_{t}(-\pi/4)|}{\sqrt{t}}\right)=O(\sqrt{t})
\]
because $|\dot{\eta}_{t}(0)|=t$. However, this growth is seen only in the $(1,3)$-entry and in the $(4,3)$-entry because all the others are scaled by a smaller exponential factor. Moreover, for $\theta \in [-\pi/3, 0]$ we have
\[
    \lambda_{12}=\lambda_{43}=\exp(2t(\cos\theta-\sin\theta))\leq \exp\left(2\sqrt{2}t-\left(\theta+\frac{\pi}{4}\right)^{2}t\right) \ ,
\]
thus we can separate the unbounded entry in $M_{t}(s)$ and write
\[
    M_{t}(s)=M_{t}^{0}(s)+\mu_{t}(s)(E_{12}+E_{43})
\]
where $M_{t}^{0}(s)=O(e^{-\beta t})$ for some $\beta>0$, $E_{12}$ and $E_{43}$ are the elementary matrices, and
\[
    \mu_{t}(s)=O\left(|\dot{\eta}_{t}(s)|\exp(-2\sqrt{2}t)\lambda_{12}\right)=O\left(\sqrt{t}e^{-(\theta+\pi/4)^{2}t}\right) \ .
\]
This upper-bound is a Gaussian function centered at $\theta=-\pi/4$, normalised such that its integral is independent of $t$. Therefore, the function $\mu_{t}(s)$ is uniformly absolutely integrable over $s \in [-\pi/3, 0]$ as $t \to +\infty$. Now, under this condition the solution to the initial value problem (\ref{eq:ODE}) satisfies (\cite[Lemma B.2]{DW})
\[
    \left\| g_{t}(\pi/4)-A_{0}R\exp\left((E_{12}+E_{43})\int_{-\frac{\pi}{4}}^{\frac{\pi}{4}}\mu_{t}(s)\right)R^{-1}A_{0}^{-1}\right\| \to 0 \ \ \text{as} \ \ t \to +\infty \ .
\]
Since $g_{t}(0)=G(\gamma_{-}(t))^{-1}G(\gamma_{0}(t)) \to L_{-}^{-1}L_{0}$, this gives the desired unipotent form. \\
The proof for $L_{0}^{-1}L_{+}$ follows the same line with the only difference given by the fact that at $\theta=\pi/4$, the leading term in the matrix $M_{t}(s)$ lies in the $(1,4)$-entry.
\end{proof}

We can now describe the boundary at infinity of the maximal surface $S$. If $\gamma$ is a stable ray in $U$ and we denote
$L_{\gamma}=\lim_{t \to +\infty}G(\gamma(t))$, then, since $F(\gamma(t))=G(\gamma(t))F_{0}(\gamma(t))$, $f(\gamma(t))$ tends to the point $p_{\gamma}$ in the boundary at infinity of $S$ that can be expressed as $p_{\gamma}=L_{\gamma}v_{\gamma}$ (see Table \ref{table:1}). Therefore, we obtain the following limit points along stable directions in the standard half-plane $U$:

\begin{table}[!h]
\begin{center}
\begin{tabular}{l l l}
\hline
\textbf{Type of path $\gamma$} & \textbf{Direction $\theta$} & \textbf{Projective limit $p_\gamma$ of $f(\gamma)$}\\
\hline
Quasi-ray& $\theta \in (-\tfrac{\pi}{2}, -\tfrac{\pi}{4})$ & $p_\gamma = L_{-}[0,-1,0,1]$\\
Quasi-ray& $ \theta \in (-\tfrac{\pi}{4}, \tfrac{\pi}{4})$ & $p_\gamma = L_{0}[1,0,1,0]$\\
Quasi-ray& $\theta\in (\tfrac{\pi}{4}, \tfrac{\pi}{2})$  & $p_\gamma = L_{+}[0,1,0,1]$\\
\hline\\
\end{tabular}
\end{center}
\caption{Limits of a polynomial maximal surface along rays}\label{table:2}
\end{table}

A direct computation, using the formulas provided by Lemma \ref{lm:unipotent}, shows that
\[
    L_{0}[0,1,0,1]=L_{+}[0,1,0,1] \ \ \ \ \text{and} \ \ \ \ 
    L_{0}[0,-1,0,1]=L_{-}[0,-1,0,1] \ .
\]
Moreover, along each unstable direction $\theta=\pm \pi/4$, the surface $S=f(\C)$ is asymptotic to the two horospherical surfaces $L_{\pm}f_{0}(\C)$ and $L_{0}f_{0}(S)$. By the same argument as above, 
\[
    L_{0}[1,s,1,s]=L_{+}[1,s,1,s] \ \ \ \ \text{and} \ \ \ \ 
    L_{0}[s,-1,s,1]=L_{-}[s,-1,s1] \ \ \ 
\]
for every $s \in \R$, hence the two horospherical surfaces which are asymptotic to $S$ in each unstable direction share one light-like segment. We deduce that the boundary at infinity of $S$ in each standard half-plane consists of two light-like segments forming the "vee" given by
\[
    L_{0}([0,1,0,1] \cup [1,s,1,s] \cup [1,0,1,0] \cup [s,-1,s,1] \cup [0,-1,0,1])  \ .
\]
Given two consecutive standard half-planes $U_{i}$ and $U_{i+1}$ we obtain two "vees" that share an extreme vertex: in fact, by considering an other standard $q$-upper-half plane $W$ that intersects $U_{i}$ and $U_{i+1}$ in a sector of angle $\pi/2$, the same argument as above shows that the direction $\pi/2$ is stable, so the ending point of the "vee" in $U_{i}$ is the same as the starting point of the "vee" in $U_{i+1}$. Allowing $i$ to vary, we assemble a map
\[
    \Gamma: \Delta_{2(k+2)} \rightarrow \partial_{\infty}S\cong S^{1}
\]
where $\Delta_{2(k+2)}$ is an abstract $1$-simplicial complex with $2(k+2)$-vertices homeomorphic to a circle. By construction $\Gamma$ is linear on each edge and its restriction to any pair of adjacent edges is an embedding, since they are mapped to segments belonging to different foliations of $Ein^{1,1}$. Therefore, $\Gamma$ is a local homeomorphism of compact, connected Hausdorff spaces, hence it is a covering. The boundary at infinity of $S$ is thus a light-like polygon and we can consider the covering as simplicial. We are only left to prove that $\Gamma$ is injective. \\
\indent We first recall that in each standard $q$-half plane $U_{i}$ we can find a quasi-ray $\gamma_{i}$ converging to the vertex in the centre of the "vee" (see Table \ref{table:2}). In addition, by the previous remark, two such quasi-rays $\gamma_{i}$ and $\gamma_{i+1}$ in consecutive standard half-planes cannot have the same limit point. \\
\indent Suppose now by contradiction that $\Gamma$ is not injective. Then the map $\Gamma$ is not injective on vertices, and we can find two distinct quasi-rays $\gamma_{i}$ and $\gamma_{j}$ that limit to the same vertex $v$ in the boundary at infinity of $S$. Necessarily $i \neq j\pm 1 \ (\text{mod} (k+2))$. Let us complete the curve $\gamma_{i}\cup \gamma_{j}$ to a simple curve $\beta$ that disconnects $\C$ into two connected components and does not meet the other standard $q$-half planes $U_{l}$ for $l \neq i,j$. Now, the quasi-rays $\gamma_{j-1}$ and $\gamma_{j+1}$ belong to two different components of $\C \setminus \beta$ and do not limit to the vertex $v$, because they are both neighbours of $\gamma_{j}$. Thus each component of $f(\C \setminus \beta)$ accumulates on at least one boundary point of $S$ that is different from $v$. This is a contradiction, because $f(\beta)$ is a properly embedded path in $S$ that limits on a single boundary point $v$ in both directions, so one of its complementary discs has $v$ as the only limit point on $\partial_{\infty}S$.\\
The proof of Proposition \ref{prop:general} is now complete. 

\subsection{The geometry of a polynomial maximal surfaces}\label{subsec:geo}
We conclude this section with some remarks about the geometry of the maximal surfaces that we have constructed. 

\begin{prop}Let $S$ be a polynomial maximal surface in $AdS_{3}$ with second fundamental form $2\Re(q)$. Then $S$ with its induced metric is quasi-isometric to $\C$ endowed with the flat metric with cone singularities $|q|$.
\end{prop}
\begin{proof} The induced metric on $S$ can be written as $I=2e^{2u}|dz|^{2}$, where $u:\C \rightarrow \R$ is the solution to Equation (\ref{eq:PDE}). By Proposition \ref{prop:vortex}, we know that
\[
    \frac{1}{2}\log(|q|) \leq u \leq \frac{1}{2}\log(|q|+C)
\]
for some positive constant $C>0$. Therefore, 
\[
    2|q| \leq 2e^{2u}|dz|^{2}\leq 2(|q|+C)
\]
and the claim follows.
\end{proof}

\begin{prop}\label{prop:princ_curv}Let $S$ be a polynomial maximal surface in $AdS_{3}$ with second fundamental form $2\Re(q)$. Then the positive principal curvature of $S$ is in $[0,1)$ and tends to $1$ for $|z| \to +\infty$.
\end{prop}
\begin{proof} By definition of $q$, the second fundamental form of $S$ is $II=2\Re(q)$. Let us denote with $\lambda$ the positive principal curvature of $S$. Then
\[
    -\lambda^{2}=\det(B)=\det(I^{-1}II)=\det(e^{-2u}\Re(q))=-e^{-4u}|q|^{2} \ .
\]
By Proposition \ref{prop:vortex}, we know $|q|<e^{2u}$, hence $\lambda<1$. \\
On the other hand, when $|z|\to +\infty$, the function $u$ diverges and the inequalities
\[
    1\leq e^{-2u}(|q|+C)=\lambda+e^{-2u}C\leq 1+e^{-2u}
\]
give the desired result. 
\end{proof}

\section{Proof of the main result}\label{sec:mainthm}
The discussion in the previous section enables us to define a map
\[
    \alpha: \mathcal{MQ}_{k} \rightarrow \mathcal{MLP}_{2(k+2)}
\]
by sending the equivalence class of a polynomial quadratic differential $[q]$ on the complex plane to the boundary at infinity of a maximal embedding $\sigma:\C \rightarrow AdS_{3}$ with second fundamental form $2\Re(q)$. This does not depend on the choice of the representative $q$, because if $q'$ is equivalent to $q$, then there exists an automorphism $T$ of $\C$ such that $T_{*}q=q'$. Therefore, the embedding $\sigma'=\sigma \circ T:\C \rightarrow AdS_{3}$ has second fundamental form $2\Re(q')$ and the boundary at infinity does not change. The main aim of this section is proving that $\alpha$ is a homeomorphism.\\
\\
\indent Let us first point out that $\alpha$ can be lifted to 
\[
    \tilde{\alpha}: \mathcal{TQ}_{k} \rightarrow \mathcal{TP}_{k+2}\times \mathcal{TP}_{k+2}
\]
sending a monic and centered polynomial quadratic differential of degree $k$ (see Proposition \ref{prop:moduli_poly}) to the two marked $(k+2)$-uples of points in $\R\Pp^{1}$ corresponding to the points of discontinuity and the images of the function $f:\R\Pp^{1} \rightarrow \R\Pp^{1}$ associated to the light-like polygon (see Proposition \ref{prop:moduli_polygo}).\\
\indent In order to see this, we need to understand how to encode the action of the finite group $\Z_{k+2}$. Given a monic polynomial quadratic differential $q$ in $\C$ of degree $k$, there are $k+2$ canonical directions corresponding to the set
\[
    D=\left\{ z \in \C \ | \arg(z)=\frac{2\pi j}{k+2} \right\} \ .
\]
Those can be understood as follows, if $q=z^{k}dz^{2}$, these are exactly the directions in which the quadratic differential takes positive real values; in the general case, these directions are characterised by the fact that they are contained eventually in a unique standard $q$-half-plane, where they correspond to quasi-rays with angle $0$. If we fix one direction $\theta_{0}=\arg(z_{0})$ with $z_{0} \in D$, we can see the action of the cyclic group $\Z_{k+2}$ as a rotation in this set. \\
\indent Let $\sigma:\C \rightarrow AdS_{3}$ be a maximal embedding associated to $q$. Let $\Delta$ denote the light-like polygon in the boundary at infinity of $S=\sigma(\C)$, and let $f_{\Delta}:\R\Pp^{1} \rightarrow \R\Pp^{1}$ be the corresponding upper-semicontinuous, locally constant, orientation-preserving function. Let us denote with $P$ and $Q$ the set of points of discontinuity and the images of $f_{\Delta}$. The direction $\theta_{0}$ gives a marking on $P$ and $Q$, that is a preferred point in each of them, in the following way. The path $\sigma(e^{i\theta_{0}}t)$ converges to a point in $\Delta$ as $t \to +\infty$. By the discussion in the previous section, the limit point is a vertex $v \in \Delta$. Its left projection gives a point $\pi_{l}(v) \in P$. The markings in the sets $P$ and $Q$ are given by selecting $p_{0}=\pi_{l}(v) \in P$ and $q_{0}=f_{\Delta}(p) \in Q$. We define
\begin{align*}   
    \tilde{\alpha}: \mathcal{TQ}_{k} &\rightarrow \mathcal{TP}_{k+2}\times \mathcal{TP}_{k+2}\\
                        q &\mapsto ((P,p_{0}),(Q,q_{0})) \ .
\end{align*}
If we change $\sigma$ to $\sigma'$ by post-composing with an isometry $A$ of $AdS_{3}$, the boundary at infinity is $\Delta'=A(\Delta)$, hence the corresponding pairs $(P',p_{0}')$ and $(Q', q_{0}')$ are equivalent to $(P,p_{0})$ and $(Q,q_{0})$.\\
\indent If we change $\sigma$ by pre-composing with the generator of the $\Z_{k+2}$-action $T(z)=\zeta_{k+2}^{-1}z$, then $\sigma'=\sigma \circ T$ is a maximal embedding with second fundamental form $2\Re(T_{*}q)$. Its boundary at infinity remains $\Delta$, but the limit point of the ray $\sigma'(e^{i\theta_{0}}t)$ changes to $v'$, which coincides with the limit point of the ray $\sigma(e^{i(\theta_{0}+2\pi/(d+2))}t)$. By the description of the limit points along rays given in Section \ref{subsec:rays}, the markings $p_{0} \in P$ and $q_{0} \in Q$ change to their successors $p_{1}$ and $q_{1}$. Therefore, $\tilde{\alpha}$ is $\Z_{k}$-equivariant.\\
\indent Finally, $\tilde{\alpha}$ is well-defined, because if $T_{*}^{j}(q)=q$ for some $j=1, \dots, k+2$, then $\sigma$ and $\sigma'=\sigma \circ T^{j}$ are maximal embeddings of $\C$ with the same embedding data, hence they differ by post-composition with an isometry. Then $\tilde{\alpha}(q)=\tilde{\alpha}(T_{*}^{j}(q))$.

\begin{prop}\label{prop:cont}The map $\tilde{\alpha}$ is continuous.
\end{prop}
\begin{proof}Let $q_{n}$ be a sequence of monic and centered polynomial quadratic differentials converging to $q$. Let $\Delta_{n}$ be a light-like polygon representing $\alpha([q_{n}])$ and $\Delta$ be a light-like polygon representing $\alpha([q])$. Let $f_{\Delta_{n}}$ and $f_{\Delta}$ be the corresponding defining functions, with points of discontinuity $P_{n}$ and $P$ and images $Q_{n}$ and $Q$. Recall that the direction $\theta_{0}$ induces markings $p_{n}\in P$, $p_{0} \in P$, $q_{n}\in Q_{n}$ and $q_{0}\in Q$.\\
\indent We need to prove that the marked sets $(P_{n}, p_{n})$ and $(Q_{n}, q_{n})$ converge to $(P,p_{0})$ and $(Q,q_{0})$, respectively. We first claim that the maximal surfaces $S_{n}$ with embedding data $I_{n}=2e^{2u_{n}}|dz|^{2}$ and $II_{n}=\Re(2q_{n})$ converge to the maximal surface $S$ with embedding data $I=2e^{2u}|dz|^{2}$ and $II=\Re(2q)$, up to isometries. In fact, since $q_{n}$ is convergent, Proposition \ref{prop:vortex} and standard Schauder estimates give a uniform bound on the $C^{1,1}$-norm of the functions $u_{n}$ on compact sets, hence $u_{n}$ weakly converges to a weak solution of
\[
    2u_{z\bar{z}}=e^{2u}-e^{-2u}|q|^{2} \ .
\]
By elliptic regularity, the limit is a strong solution, and by uniqueness it must coincide with $u$. Therefore, $u_{n}$ converges to $u$ smoothly on compact sets. This implies that the embedding data of $S_{n}$ converges to the embedding data of $S$, thus $S_{n}$ converges to $S$ up to changing $S_{n}$ by global isometries of $AdS_{3}$.  \\
\indent We deduce that $\Delta_{n} \to \Delta$, and  that $P_{n} \to P$ and $Q_{n}\to Q$. Since $\sigma_{n}$ converges to $\sigma$ smoothly on compact sets, the limit points of the rays $\sigma_{n}(e^{i\theta_{0}}t)$ converge to the limit point of the ray $\sigma(e^{i\theta_{0}}t)$, hence $p_{n} \to p_{0}$ and $q_{n} \to q_{0}$. 
\end{proof}

\indent In order to prove the injectivity of the map $\tilde{\alpha}$, the following lemma is crucial:
\begin{lemma}\label{lm:unique_max} Let $\Delta \subset Ein^{1,1}$ be a light-like polygon. If there exists a maximal surface $S\subset AdS_{3}$ bounding $\Delta$, then it is unique. 
\end{lemma}
\begin{proof} We can use the same argument appeared in \cite{BTT} at infinity.\\
\indent Suppose, by contradiction, that there exists another maximal surface $S'$ with boundary at infinity $\Delta$. We choose $\widehat{S}$ and $\widehat{S'}$ their lifts to $\widehat{AdS}_{3}$ in such a way that they share the same boundary at infinity. As a consequence, the function
\begin{align*}
        B: \widehat{S} \times \widehat{S'} &\rightarrow \R \\
            (u,v) &\mapsto \langle u , v \rangle 
\end{align*}
is always non-positive (see \cite[Lemma 3.24]{BTT}). Notice that $B$ is related to the time-like distance in $\widehat{AdS}_{3}$. In fact:
\begin{itemize}
    \item $|B(u,v)|>1$ if and only if $u$ and $v$ are connected by a space-like geodesic;
    \item $|B(u,v)|=1$ if and only if $u$ and $v$ are connected by a light-like geodesic;
    \item $|B(u,v)|<1$ if and only if $u$ and $v$ are connected by a time-like geodesic.
\end{itemize}
Moreover, in the last case the time-like distance between $u$ and $v$ is 
\[
    d(u,v)=\arccos(-B(u,v)) \ .
\]
If $S$ and $S'$ are different, then there exists a pair of points $(u_{0},v_{0}) \in \widehat{S} \times \widehat{S'}$ such that $B(u_{0},v_{0})>-1$ (see \cite[Lemma 3.25]{BTT}). It was proved in \cite[Theorem 3.13]{BTT} that a point of maximum of $B$ gives a contradiction.\\
\indent In our context, however, we cannot assume in general that $B$ takes its maximum, but we can apply a similar argument at infinity. Given a point $u \in \widehat{S}$, we first notice that there exists a point $v_{u} \in \widehat{S'}$ that realises
\[
    \overline{B}=\sup \{B(u,v) \ | \ v \in \widehat{S'} \} \ :
\]
in fact, since $u$ lies in the domain of dependence of the boundary at infinity of $\widehat{S'}$, the light-cone centered at $u$ intersects $\widehat{S'}$ in a compact set, and the supremum of $B(u, \cdot)$ is achieved in this set. Now, let us fix $x_{0} \in AdS_{3}$ and $\nu \in T_{x_{0}}AdS_{3}$. If $u_{n} \in \widehat{S}$ is a sequence of points such that 
\[
    \lim_{n \to +\infty} B(u_{n}, v_{u_{n}})= \sup(B) \ 
\]    
and $\nu_{n}$ is the sequence of future-oriented unit normal vectors to $\widehat{S}$ at $u_{n}$, there exist isometries $g_{n} \in \Isom(\widehat{AdS}_{3})$ such that $g(u_{n})=x_{0}$ and $d_{u_{n}}g_{n}(\nu_{n})=\nu$. General properties about the behaviour of constant mean curvature surfaces in $AdS_{3}$ (see \cite[Lemma 5.1]{bon_schl} or \cite[Lemma 4.1]{TambuCMC}) imply that the sequence of maximal surfaces $\widehat{S}_{n}=g_{n}(\widehat{S})$ converges in a neighbourhood of $x_{0}$ to a maximal surface $\widehat{S}_{\infty}$. In addition, the maximal surfaces $\widehat{S}_{n}'=g_{n}(\widehat{S}_{n}')$ also converge to a limit $\widehat{S}_{\infty}'$ in a neighbourhood of its intersection with the future oriented normal time-like geodesic to $\widehat{S}_{\infty}$ at $x_{0}$, because the sequence $u_{n}$ converges to the supremum of the time-like distance between $\widehat{S}$ and $\widehat{S'}$, and we can choose $u_{n}$ so that the  norm of the differential at $u_{n}$ of the restriction to $\widehat{S}$ of the distance to $\widehat{S'}$ goes to zero as $n \to +\infty$ (\cite{Yau}). But now the function $B$ defined on $\widehat{S}_{\infty} \times \widehat{S'}_{\infty}$ achieves a maximum bigger than $-1$ and this gives a contradiction.
\end{proof}

\begin{prop}\label{prop:inj} The map $\tilde{\alpha}$ is injective.
\end{prop}
\begin{proof} Let $q,q' \in \mathcal{TQ}_{k}$ be different monic and centered quadratic differentials. If there exists $j \in \{1, \dots, k\}$ such that $q'=T_{*}^{j}q$, where $T(z)=\zeta_{k+2}z$ is a generator of the $\Z_{k+2}$-action, then  the equivariance of the map already implies that $\tilde{\alpha}(q) \neq \tilde{\alpha}(q')$. \\
Otherwise, suppose by contradiction that $\tilde{\alpha}(q)=\tilde{\alpha}(q')$. Then, we can choose maximal surfaces $S$ and $S'$ with second fundamental form $2\Re(q)$ and $2\Re(q')$ with the same boundary at infinity $\Delta$. By Lemma \ref{lm:unique_max}, $S$ and $S'$ must coincide, and, in particular have the same embedding data. Therefore, there exists a biholomorphism $T'$ of $\C$ such that $T'_{*}q'=q$, but this is impossible because $q$ and $q'$ do not lie in the same $\Z_{k+2}$-orbit and they are both monic and centered.
\end{proof}

\subsection{Properness} The proof of the properness of the map $\alpha$ follows the line of the proof of Theorem 3.5 in \cite{TambuCMC}. We first recall some preliminary results.\\
\\
\indent We identify $\widehat{AdS}_{3}$ with $\h^{2} \times S^{1}$ and we denote with $\widetilde{AdS}_{3}$ its Universal cover. In this setting, the formula for the mean curvature $H$ of a space-like surface $S$ that is the graph of a function $h:\h^{2} \rightarrow \R$ is well-known (\cite{Bartnik}):
\begin{equation}\label{eq:PDEmax}
    H=\frac{1}{2v_{S}}(\dive_{S}(\chi\grad_{S}h)+\dive_{S}T) \ ,
\end{equation}
where, denoting with $t$ the $\R$-coordinate, $\chi^{2}=-\| \frac{\partial}{\partial t}\|^{2}$, $T$ is the unitary time-like vector field that gives the time-like orientation and $v_{S}=-\langle \nu , T\rangle$, where $\nu$ is the future-directed unit normal vector field to $S$.\\
\\
\indent Given a point $p \in \widetilde{AdS}_{3}$, we denote with $I^{+}(p)$ the future of $p$ and with $I^{+}_{\epsilon}(p)$ the points in the future of $p$ at distance at least $\epsilon$. The key a-priori estimate is the following:
\begin{lemma}\cite[Lemma 4.13]{bon_schl}\label{lm:estimate} Let $p \in \widetilde{AdS}_{3}$ and $\epsilon>0$. Let $K$ be a compact domain contained in a region where the covering map $\widetilde{AdS}_{3}\rightarrow AdS_{3}$ is injective. There exists a constant $C=C(p, \epsilon, K)$ such that for every maximal surface $M$ that verifies
\begin{itemize}
    \item $\partial M \cap I^{+}(p)=\emptyset$;
    \item $M \cap I^{+}(p) \subset K$,
\end{itemize}
we have that
\[
    \sup_{M \cap I_{\epsilon}^{+}(p)}v_{M}\leq C \ .
\]
\end{lemma}

\begin{oss}The same conclusion holds if we consider the past $I^{-}(p)$ of a point $p \in \widetilde{AdS}_{3}$. In fact, being maximal does not depend on the choice of the time orientation, and $v_{S}$ is invariant under the change of time-orientation.
\end{oss}

\begin{prop}\label{prop:proper}The map $\alpha$ is proper.
\end{prop}
\begin{proof} We first prove that if $\Delta_{n}$ is a sequence of light-like polygons converging to $\Delta$ in the Hausdorff topology and $S_{n}$ are maximal surfaces with boundary at infinity $\Delta_{n}$, then $S_{n}$ converges $C^{2}$ on compact sets to a maximal surface $S$ with boundary at infinity $\Delta$. \\
\indent Let $\widetilde{\mathcal{C}(\Delta)}$ be a lift of the convex hull of $\Delta$ to the Universal cover. For any point $\tilde{p} \in \widetilde{\mathcal{D}(\Delta)} \cap I^{+}(\partial^{-}\widetilde{\mathcal{C}(\Delta)})$, we choose $\epsilon(\tilde{p})>0$ so that the family 
\[
    \{I^{+}_{\epsilon(\tilde{p})}(\tilde{p}) \cap \widetilde{\mathcal{C}(\Delta)}\} \cup \{ I^{-}_{\epsilon(\tilde{p})}(\tilde{p}) \cap \widetilde{\mathcal{C}(\Delta)}\}
\]
is an open covering of $\widetilde{\mathcal{C}(\Delta)}$. Since $\Delta_{n}$ converges to $\Delta$ in the Hausdorff topology, the convex hull of $\Delta_{n}$ converges to the convex hull of $\Delta$, thus we can find $n_{0}$ such that the above family of open sets provide an open covering of
\[
    K=\overline{\bigcup_{n \geq n_{0}} \widetilde{\mathcal{C}(\Delta_{n})}} \ .
\]
Given a number $R>0$, we denote with $B_{R}$ the ball of radius $R$ in $\h^{2}$ centered at the origin in the Poincar\'e model. The intersection $(B_{R} \times \R) \cap K$ is compact, so there is a finite number of points $\tilde{p}_{1}, \dots, \tilde{p}_{m}$ such that
\[
    (B_{R}\times \R)\cap K \subset \bigcup_{j=1}^{m}I_{\epsilon(\tilde{p}_{j})}^{+}(\tilde{p}_{j}) \cup \bigcup_{j=1}^{m}I^{-}_{\epsilon(\tilde{p}_{j})}(\tilde{p}_{j}) \ .
\]
We notice that, since $\tilde{p}_{j} \in \widetilde{\mathcal{D}(\Delta)}$, the intersections $I^{\pm}(\tilde{p}_{j}) \cap \widetilde{\mathcal{D}(\Delta)}$ are compact. Moreover, since the plane dual to $\tilde{p}_{j}$ is disjoint from $\Delta$ for every $j=1, \dots, m$, if we choose $n_{0}$ big enough, the same is true for $\Delta_{n}$ for every $n \geq n_{0}$, because $\Delta_{n}$ converges to $\Delta$ in the Hausdorff topology. In this way, we can ensure that the sets $K_{j}^{+}=\overline{I^{+}(\tilde{p}_{j})}\cap K$ and $K_{j}^{-}=\overline{I^{-}(\tilde{p}_{j})}\cap K$ are compact and contained in a region where the covering map $\pi: \widetilde{AdS}_{3} \rightarrow AdS_{3}$ is injective. By Lemma \ref{lm:estimate}, there are constants $C_{j}^{\pm}$ such that
\[
    \sup_{M\cap I^{\pm}_{\epsilon(\tilde{p}_{j})}(\tilde{p}_{j})}v_{M} \leq C_{j}^{\pm}
\]
for every maximal surface $M$ satisfying:
\begin{enumerate}
    \item [i)] $\partial M \cap I^{\pm}(\tilde{p}_{j})=\emptyset$;
    \item [ii)] $M \cap I^{\pm}(\tilde{p}_{j})$ is contained in $K_{j}^{\pm}$.
\end{enumerate}
We can apply this to our sequence of maximal surfaces $\widetilde{S}_{n}$ for $n \geq n_{0}$, and we obtain that 
\[
    \sup_{\widetilde{S}_{n}\cap(B_{R}\times \R)}v_{\widetilde{S}_{n}} \leq \max\{C_{1}^{\pm}, \dots, C_{m}^{\pm}\}
\]
for every $n \geq n_{0}$. We deduce that for every $R>0$, there exists a constant $C(R)$ such that $v_{\widetilde{S}_{n}}$ is bounded by $C(R)$ for $n$ sufficiently large. If $\widetilde{S}_{n}$ are graphs of the functions $h_{n}:\h^{2} \rightarrow \R$, comparing the previous estimate with Equation (\ref{eq:PDEmax}), we see that the restriction of $h_{n}$ on $B_{R}$ is the solution of a uniformly elliptic quasi-linear PDE with bounded coefficients. Since $h_{n}$ and the gradient of $h_{n}$ are uniformly bounded (Proposition \ref{prop:graph}), by elliptic regularity, the norms of $h_{n}$ in $C^{2,\alpha}(B_{R-1})$ are uniformly bounded. We can thus extract a subsequence $h_{n_{k}}$ that converges $C^{2}$ to some function $h$ on compact sets. Since $h$ is the $C^{2}$ limit of solutions of Equation (\ref{eq:PDEmax}), it is still a solution and its graph $\widetilde{S}$ is a maximal surface. When projecting back to $AdS_{3}$, the boundary at infinity of $S$ coincides with $\Delta$ because it is the Hausdorff limit of the curves $\Delta_{n}$, which converge to $\Delta$ by construction. Notice that the embedding data of $S_{n}$ converge on compact sets to the embedding data of $S$.\\
\indent We can now conclude that the map $\alpha$ in proper. Let $\Delta_{n}=\alpha([q_{n}])$ be convergent to a light-like polygon $\Delta$ in $\mathcal{MLP}_{2(k+2)}$. Then, up to acting with elements of $\dPSL$, we can assume that $\Delta_{n}$ converges to $\Delta$ in the Hausdorff topology. By the previous discussion, the sequence $q_{n}$ must converge uniformly on compact sets to a holomorphic quadratic differential $q$ up to biholomorphisms of $\C$, because their real part determines the second fundamental form of the maximal surfaces bounding $\Delta_{n}$. Since all $q_{n}$ are polynomial of fixed degree $k$, $q$ is necessarily a polynomial of degree at most $k$. But, the degree of the polynomial determines the number of edges of the light-like polygon at infinity, hence the degree is exactly $k$ and the proof is complete using Proposition \ref{prop:poly_top}.
\end{proof}

\begin{proof}[Proof of Theorem \ref{thm:C}] By Proposition \ref{prop:cont} and Proposition \ref{prop:inj}, the map $\tilde{\alpha}$ is continuous and injective. By the Domain Invariance Theorem, $\tilde{\alpha}$ is open. Since it is equivariant, its projection $\alpha$ is open, as well. Since a proper map between locally compact topological spaces is closed, the image of $\alpha$ is a connected component of $\mathcal{MLP}_{2(k+2)}$. Since the latter is connected, $\alpha$ is surjective, hence a homeomorphism.
\end{proof}

\begin{proof}[Proof of Theorem \ref{thm:B}] Since the map $\alpha$ is surjective, every light-like polygon is the boundary at infinity of a space-like maximal surface. Uniqueness follows from Proposition \ref{lm:unique_max}.
\end{proof}

\section{Application}\label{sec:application}
Let $\Omega_{l}, \Omega_{r} \subset \h^{2}$ be open domains of the hyperbolic plane. An orientation preserving diffeomorphism $m:\Omega_{l}\rightarrow \Omega_{r}$ is minimal Lagrangian if its graph is a minimal surface in $\h^{2} \times \h^{2}$ that is Lagrangian for the symplectic form $\omega_{\h^{2}}\oplus -\omega_{\h^{2}}$.\\
\\
\indent Minimal Lagrangian maps have been extensively studied when $\Omega_{r}=\Omega_{l}=\h^{2}$. For instance, if we require $m$ to be equivariant under the action of two Fuchsian representations $\rho_{r},\rho_{l}:\pi_{1}(\Sigma)\rightarrow \PSL(2,\R)$ of the fundamental group of a closed, oriented and connected surface of genus at least two, a result by Schoen (\cite{Schoenharmonic}) states that such an $m$ always exists and is unique. Later, Bonsante and Schlenker (\cite{bon_schl}) used anti-de Sitter geometry to construct minimal Lagrangian maps from $\h^{2}$ to $\h^{2}$ and extended Schoen result in the following sense: given a quasi-symmetric homeomorphism of the circle, there exists a unique minimal Lagrangian extension to the hyperbolic plane. Properties of these maps have been then studied by Seppi (\cite{seppimaximal}). \\
\\
\indent Here we use the techniques introduced by Bonsante and Schlenker in order to construct a particular class of minimal Lagrangian maps between ideal polygons in $\h^{2}$. Let us first recall how their construction works. Let $S$ be a maximal surface in anti-de Sitter space. We denote with $q$ the holomorphic quadratic differential such that $II=2\Re(q)$. The Gauss map
\[
    G: S \rightarrow \h^{2} \times \h^{2}
\]
is harmonic, and the two projections are also harmonic 
\[
    \Pi_{r}, \Pi_{l}:S \rightarrow \h^{2}  
\]
with opposite Hopf differentials $\pm 2iq$ (e.g. \cite[Prop. 6.3]{TambuCMC}). If we assume that $S$ has principal curvatures in $(-1,1)$, and that these maps are diffeomorphisms from $S$ to open domains $\Omega_{r,l}$ of the hyperbolic plane, then the composition
\[
    \Pi_{r} \circ \Pi_{l}^{-1}:\Omega_{l} \rightarrow \Omega_{r}
\]
is minimal Lagrangian. We remind that if the principal curvatures are in $(-1,1)$, the left and right Gauss maps are always orientation preserving local diffeomorphisms, but global injectivity may fail. The conformal class of the induced metric on the maximal surface is called the centre of the minimal Lagrangian map. By the aforementioned result of Bonsante and Schlenker, every (equivariant) minimal Lagrangian map from $\h^{2}$ to $\h^{2}$ can be obtained by this procedure and the centre is a hyperbolic surface. \\  
\\
\indent As a consequence of our study about polynomial maximal surfaces in anti-de Sitter space, we can construct a particular class of minimal Lagrangian maps between ideal polygons: 

\begin{defi}Let $P,Q$ be ideal polygons in $\h^{2}$ with the same number of vertices. We say that a minimal Lagrangian map $m:P \rightarrow Q$ factors through the complex plane if there exist two complete harmonic diffeomorphisms $f:\C \rightarrow P$ and $f':\C \rightarrow Q$ with opposite Hopf differentials such that $m=f'\circ f^{-1}$.
\end{defi}

\begin{oss} We remark that the condition on the number of vertices is necessary for the existence of a minimal Lagrangian map, as it preserves the volume. 
\end{oss}

\begin{oss}We recall that a harmonic diffeomorphism $f:\C \rightarrow \h^{2}$ is said to be complete if the metric $\|\partial f\|^{2}|dz|^{2}$ is complete. This is a technical assumption in order to use the results of \cite{polygons_hyp}. We do not know if there exist harmonic diffeomorphisms from the complex plane to an ideal polygon whose Hopf differential is not a polynomial.
\end{oss} 

Let us now show that the construction of Bonsante and Schlenker applied to a polynomial maximal surface $S$ in $AdS_{3}$ induces a minimal Lagrangian map between ideal polygons. We first remark that by Proposition \ref{prop:princ_curv} the principal curvatures of $S$ are in $(-1,1)$. Moreover, the left- and right- Gauss maps are harmonic with polynomial Hopf differentials. Under this condition, the metrics $\|\partial \Pi_{r}\|^{2}|dz|^{2}$ and $\|\partial \Pi_{l}\|^{2}|dz|^{2}$ are complete and the harmonic maps are diffeomorphisms onto their image (\cite[Theorem 5.1]{QL_vortex}). Now, complete harmonic maps with polynomial Hopf differentials of degree $k$ send diffeomorphically the complex plane to an ideal polygon with $k+2$ vertices (\cite{polygons_hyp}). These polygons are related to the boundary at infinity of $S$ in a precise way:

\begin{prop}\label{prop:max_poly}Let $\Delta$ be the boundary at infinity of a polynomial maximal surface $S$ in $AdS_{3}$. Let $f_{\Delta}:\R\Pp^{1} \rightarrow \R\Pp^{1}$ be the corresponding function. If $P$ and $Q$ are the sets of points of discontinuity and in the image of $f_{\Delta}$, then $S$ induces a minimal Lagrangian map between the ideal polygons with vertices $P$ and $Q$.
\end{prop}
\begin{oss}Recall that in the definition of $f_{\Delta}$ we have fixed a totally geodesic space-like plane $P_{0}$ in $AdS_{3}$. Also the definitions of the projections $\Pi_{r, l}$ depend on the choice of a totally geodesic space-like plane. The above result holds if these choices are compatible, i.e. we use the same space-like plane $P_{0}$. 
\end{oss}
\begin{proof} We do the proof for the left Gauss map $\Pi_{l}:S \rightarrow \h^{2}$, the other case being analogous. We recall that the set $P$ is obtained by projecting the edges of $\Delta$ that lie in the left-foliation to the boundary at infinity of $P_{0}$ via $\pi_{l}$. Since we already know that the image of $\Pi_{l}$ is an ideal polygon with a precise number of vertices, it is sufficient to prove that if $p_{n} \in S$ converges to $p_{\infty} \in e_{l}$, where $e_{l}$ is an edge of $\Delta$ belonging to the left foliation, then $\Pi_{l}(p_{n})$ converges to $\pi_{l}(e_{l})$. Now, a polynomial maximal surface is asymptotic to horospherical surfaces at infinity, hence it is enough to prove this for horospherical surfaces. Moreover, it is sufficient to prove it for the standard horospherical surface $S_{0}$ described in Section \ref{subsec:horo}, because all others are obtained from $S_{0}$ by acting with an isometry $(A,B)\in \dPSL$ and the projections $\pi_{l}$ and $\Pi_{l}$ both change by post-composition with $A$. \\
\indent The proof now boils down to a standard computation in anti-de Sitter geometry. For this purpose we use the identification between $AdS_{3}$ and $\PSL(2, \R)$ given by
\[
    (x_{0},x_{1},x_{2},x_{3}) \mapsto \begin{pmatrix} x_{3}-x_{1} & x_{0}-x_{2}\\
                                                      x_{0}+x_{2} & x_{3}+x_{1} 
                                       \end{pmatrix} .
\]
Morover, we choose as $P_{0}$ the space-like plane
\[
    P_{0}=\{A \in \PSL(2,\R) \ | \ \trace(A)=0\} \ .
\]
In this model the horospherical surface $S_{0}$ is parameterised by 
\[
    f_{0}(x,y)=\frac{1}{\sqrt{2}}\begin{pmatrix} e^{-2y} & -e^{-2x}\\
                                                 e^{2x} & e^{2y} 
                                 \end{pmatrix} \ .
\]
The left Gauss map can be computed as follows (\cite{folKsurfaces}): let $p \in S_{0}$ and $T_{p}S_{0}$ be the totally geodesic plane tangent to $S$ at $p$. Let $G(p)$ be the point dual to $T_{p}S_{0}$. Then the isometry $(\Id, G(p))$ sends $T_{p}S$ to $P_{0}$ and the left Gauss map is given by
\[
    \Pi_{l}(p)=p\cdot G(p)^{-1}
\]
where we are thinking of $p$ as an element of $\PSL(2,\R)$ itself. In our setting
\[
    G(x,y)=\frac{1}{2}\begin{pmatrix} -e^{-2y} & -e^{-2x} \\
                                        e^{2x} & -e^{2y}
                      \end{pmatrix} 
\]
and the left Gauss map is
\[
    \Pi_{l}(x,y)=f_{0}(x,y)\cdot G(x,y)^{-1}=\begin{pmatrix} 0 & e^{-2x-2y} \\
                                                  -e^{2x+2y} & 0 
                                  \end{pmatrix} \ .
\]
The edges of the light-like polygon at infinity belonging to the left-foliations are (see Table \ref{table:1})
\[
    e_{1}=\begin{pmatrix} 0 & 0 \\
                          1 & s
          \end{pmatrix}  \ \ \ \ \ \ \ \text{and} \ \ \ \ \ \ \  e_{2}=\begin{pmatrix} s & -1 \\
                                                                                        0 & 0
                                                                        \end{pmatrix} \ ,
\]
where those matrices are to be intended up to scalar multiples. The left-projection is then performed by taking their intersection with the boundary at infinity of $P_{0}$ of the projective lines they generate, i.e.
\[
    \pi_{l}(e_{1})=\begin{pmatrix} 0 & 0 \\
                                    1 & 0 
                    \end{pmatrix} \ \ \ \ \ \ \ \text{and} \ \ \ \ \ \ \ \pi_{l}(e_{2})=\begin{pmatrix} 0 & 1 \\
                                                                                                        0 & 0
                                                                                         \end{pmatrix} \ .
\]
Now, the edge $e_{1}$ is obtained as projective limit along the rays $\gamma_{1}(t)=e^{i\pi/4}t+iz$ for $z \in \R$, and the above computations show that
\[
    \lim_{t \to +\infty}\Pi_{l}(f_{0}(\gamma_{1}(t))=\begin{pmatrix} 0 & 0 \\
                                                            1 & 0 
                                            \end{pmatrix}=\pi_{l}(e_{1}) \ .
\]
The computation for the edge $e_{2}$ is analogous. 
\end{proof}

\begin{prop}\label{prop:unique}Let $m:P \rightarrow Q$ be a minimal Lagrangian map between ideal polygons with $k\geq 3$ vertices that factors through the complex plane. Then $m$ is induced by a polynomial maximal surface in anti-de Sitter space.
\end{prop}
\begin{proof} By assumption the map $m$ can be written as $m=f'\circ f^{-1}$, where $f:\C \rightarrow P$ and $f':\C \rightarrow Q$ are complete harmonic maps with opposite Hopf differentials $\pm 2iq$. By \cite[Theorem 1.1]{polygons_hyp}, the Hopf differential must be a polynomial of degree $k-2$. Therefore, associated to $q$ is a polynomial maximal surface $S$, conformally equivalent to $\C$, in anti-de Sitter space. The left and right projections of the Gauss map of $S$ are harmonic with Hopf differentials $\pm 2iq$, hence they must coincide with $f$ and $f'$ (up to post-composition with elements of $\PSL(2,\R))$. Therefore, up to changing $S$ by a global isometry of $AdS_{3}$, the left and right Gauss map coincide with $f$ and $f'$, and $m$ is induced by $S$.
\end{proof}

\begin{proof}[Proof of Theorem \ref{thm:A}] Given two ideal polygons $P$ and $Q$ with $k\geq 3$ vertices, there are at most $k$ distinct light-like polygons in the Einstein Universe $\Delta$ whose defining functions $f_{\Delta}$ have points of discontinuity $P$ and image $Q$: in fact those are the elements in the fibre of the map
\[
    \mathcal{MLP}_{2k}\cong (\mathcal{TP}_{k} \times \mathcal{TP}_{k})/\Z_{k} \rightarrow \mathcal{TP}_{k}/\Z_{k}\times \mathcal{TP}_{k}/\Z_{k}
\]
associating to a pair of labelled ideal polygons their equivalence classes in the moduli space of ideal polygons. By Theorem \ref{thm:B}, each light-like polygon bounds a polynomail maximal surface that induces a minimal Lagrangian map between $P$ and $Q$ (Proposition \ref{prop:max_poly}). By Proposition \ref{prop:unique}, every such minimal Lagrangian map comes from a polynomial maximal surface in anti-de Sitter space, and the theorem follows.
\end{proof}

\bibliographystyle{alpha}
\bibliography{bs-bibliography}

\begin{thebibliography}{HTTW95}

\bibitem[Bar88]{Bartnik}
Robert Bartnik.
\newblock Regularity of variational maximal surfaces.
\newblock {\em Acta Math.}, 161(3-4):145--181, 1988.

\bibitem[BB04]{BB_wild}
Olivier Biquard and Philip Boalch.
\newblock Wild non-abelian {H}odge theory on curves.
\newblock {\em Compos. Math.}, 140(1):179--204, 2004.

\bibitem[BBS11]{bbsads}
Thierry Barbot, Francesco Bonsante, and Jean-Marc Schlenker.
\newblock Collisions of particles in locally {A}d{S} spacetimes {I}. {L}ocal
  description and global examples.
\newblock {\em Comm. Math. Phys.}, 308(1):147--200, 2011.

\bibitem[BBZ11]{folKsurfaces}
Thierry Barbot, Fran{\c{c}}ois B{\'e}guin, and Abdelghani Zeghib.
\newblock Prescribing {G}auss curvature of surfaces in 3-dimensional
  spacetimes: application to the {M}inkowski problem in the {M}inkowski space.
\newblock {\em Ann. Inst. Fourier (Grenoble)}, 61(2):511--591, 2011.

\bibitem[Bre08]{brendle2008}
S.~Brendle.
\newblock Minimal lagrangian diffeomorphisms between domains in the hyperbolic
  plane.
\newblock {\em J. Differential Geom.}, 80(1):1--22, 09 2008.

\bibitem[BS10]{bon_schl}
Francesco Bonsante and Jean-Marc Schlenker.
\newblock Maximal surfaces and the universal {T}eichm\"uller space.
\newblock {\em Invent. Math.}, 182(2):279--333, 2010.

\bibitem[BST17]{volumeAdS}
Francesco Bonsante, Andrea Seppi, and Andrea Tamburelli.
\newblock On the volume of anti--de {S}itter maximal globally hyperbolic
  three-manifolds.
\newblock {\em Geom. Funct. Anal.}, 27(5):1106--1160, 2017.

\bibitem[BZ88]{saddle}
Yu.~D. Burago and V.~A. Zalgaller.
\newblock {\em Geometric inequalities}, volume 285 of {\em Grundlehren der
  Mathematischen Wissenschaften [Fundamental Principles of Mathematical
  Sciences]}.
\newblock Springer-Verlag, Berlin, 1988.
\newblock Translated from the Russian by A. B. Sosinski\u\i , Springer Series
  in Soviet Mathematics.

\bibitem[CTT17]{BTT}
Brian Collier, Nicolas Tholozan, and J\'er\'emy Toulisse.
\newblock The geometry of maximal representations of surface groups into
  ${SO}_{0}(2,n)$.
\newblock {\em arXiv:1702.08799}, 2017.

\bibitem[DW15]{DW}
David Dumas and Michael Wolf.
\newblock Polynomial cubic differentials and convex polygons in the projective
  plane.
\newblock {\em Geom. Funct. Anal.}, 25(6):1734--1798, 2015.

\bibitem[ES64]{eells}
J.~Eells and J.H. Sampson.
\newblock Harmonic mappings of {R}iemannian manifolds.
\newblock {\em Amer. J. Math}, 86:109--159, 1964.

\bibitem[HTTW95]{polygons_hyp}
Zheng-Chao Han, Luen-Fai Tam, Andrejs Treibergs, and Tom Wan.
\newblock Harmonic maps from the complex plane into surfaces with nonpositive
  curvature.
\newblock {\em Comm. Anal. Geom.}, 3(1-2):85--114, 1995.

\bibitem[Li19]{QL_vortex}
Qiongling Li.
\newblock On the uniqueness of vortex equations and its geometric applications.
\newblock {\em J. Geom. Anal.}, 29(1):105--120, 2019.

\bibitem[Sch93]{Schoenharmonic}
Richard~M. Schoen.
\newblock The role of harmonic mappings in rigidity and deformation problems.
\newblock In {\em Complex geometry ({O}saka, 1990)}, volume 143 of {\em Lecture
  Notes in Pure and Appl. Math.}, pages 179--200. Dekker, New York, 1993.

\bibitem[Sep16]{seppimaximal}
Andrea Seppi.
\newblock Maximal surfaces in {A}nti-de {S}itter space, width of convex hulls
  and quasiconformal extensions of quasisymmetric homeomorphisms.
\newblock {\em To appear in Journal of the EMS}, 2016.

\bibitem[Str84]{Strebel}
Kurt Strebel.
\newblock {\em Quadratic differentials}, volume~5 of {\em Ergebnisse der
  Mathematik und ihrer Grenzgebiete (3) [Results in Mathematics and Related
  Areas (3)]}.
\newblock Springer-Verlag, Berlin, 1984.

\bibitem[Tam17]{entropy}
Andrea Tamburelli.
\newblock Entropy degeneration of globally hyperbolic maximal compact anti-de
  sitter structures.
\newblock {\em arXiv:1710.05827}, 2017.

\bibitem[Tam18]{regularAdS}
Andrea Tamburelli.
\newblock Regular globally hyperbolic maximal anti-de {S}itter structures.
\newblock {\em arXiv:1806.08176}, 2018.

\bibitem[Tam19a]{TambuCMC}
Andrea Tamburelli.
\newblock Constant mean curvature foliation of domains of dependence in
  {$AdS_3$}.
\newblock {\em Trans. Amer. Math. Soc.}, 371(2):1359--1378, 2019.

\bibitem[Tam19b]{degeneration}
Andrea Tamburelli.
\newblock Degeneration of globally hyperbolic maximal anti-de {S}itter
  structures along pinching sequences.
\newblock {\em Differential Geom. Appl.}, 64:125--135, 2019.

\bibitem[Tam19c]{wildAdS}
Andrea Tamburelli.
\newblock Wild globally hyperbolic maximal anti-de {S}itter structures.
\newblock {\em arXiv:1901.00129}, 2019.

\bibitem[Tei60]{Teichmap}
O.~Teichm\"uller.
\newblock Extremale quasikonforme {A}bbildungen und quadratische {D}if-
  ferentiale.
\newblock {\em Abh. Preuss. Akad. Wiss. Math.-Nat. Kl.}, 26:126--155, 1960.

\bibitem[Tro92]{Tromba_book}
Anthony Tromba.
\newblock {\em Teichm\"uller {T}heory in {R}iemannian {G}eometry}.
\newblock Lectures in {M}athematics. {ETH} {Z}\"urich. Birkh\"auser {B}asel,
  1992.

\bibitem[TV95]{oneharmonic}
Stefano Trapani and Giorgio Valli.
\newblock One-harmonic maps on {R}iemann surfaces.
\newblock {\em Comm. Anal. Geom.}, 3(3-4):645--681, 1995.

\bibitem[TW19]{TW}
Andrea Tamburelli and Michael Wolf.
\newblock Planar minimal surfaces with polynomial growth in the
  {S}{P}(4,{R})-symmetric space.
\newblock {\em In preparation}, 2019.

\bibitem[Wol89]{Wolf_harmonic}
Michael Wolf.
\newblock The {T}eichm\"uller theory of harmonic maps.
\newblock {\em J. Differential Geom.}, 29(2):449--479, 1989.

\bibitem[Yau75]{Yau}
Shing-Tung Yau.
\newblock Harmonic functions on complete riemannian manifolds.
\newblock {\em Communications on Pure and Applied Mathematics}, 28(2):201--228,
  1975.

\end{thebibliography}

\bigskip

\noindent \footnotesize \textsc{DEPARTMENT OF MATHEMATICS, RICE UNIVERSITY}\\
\emph{E-mail address:}  \verb|andrea.tamburelli@uni.lu|

\end{document}